\documentclass[11,reqno]{amsart}
\usepackage{amsmath}
\usepackage{amsfonts}
\usepackage{amssymb}
\usepackage{verbatim}
\usepackage{epsfig}  		
\usepackage{epic,eepic}       
\usepackage{tikz}
\usepackage{tikz-cd}
\usetikzlibrary{matrix, calc, arrows}
\usepackage{hyperref}

\usepackage {amsmath,amssymb,verbatim,geometry}
\newtheorem{thm}{Theorem}[section]
\newtheorem{cor}[thm]{Corollary}
\newtheorem{lem}[thm]{Lemma}
\newtheorem{mydef}[thm]{Definition}
\newtheorem{rem}[thm]{Remark}
\newtheorem{ex}[thm]{Example}

\newtheorem{prop}[thm]{Proposition}

\newcommand{\ric}[1]{\text{Ric}(#1)}

\newcommand{\xRightarrow}[2][]{\ext@arrow 0359\Rightarrowfill@{#1}{#2}}

\DeclareMathOperator{\Aut}{Aut}

\DeclareMathOperator{\DF}{DF}
\DeclareMathOperator{\proj}{pr}
\DeclareMathOperator{\NA}{NA}

\newcommand{\R}{\mathbb{R}}
\newcommand{\C}{\mathbb{C}}

\renewcommand{\epsilon}{\varepsilon}

\newcommand{\ddb}{i\partial \bar\partial}

\newcommand{\mft}{\mathfrak{t}}

\newcommand{\A}{\mathcal{A}}
\newcommand{\X}{\mathcal{X}}

\begin{document}
\title[K-polystability of cscK manifolds]{On K-polystability of cscK manifolds with transcendental cohomology class}
\author{Zakarias Sj\"ostr\"om Dyrefelt \\ \\ With an appendix by Ruadha\'i Dervan}

\address{Zakarias Sj\"ostr\"om Dyrefelt, Department of Mathematical Sciences, Chalmers University of Technology, Chalmers Tv\"argata 3, SE-412 96, G\"oteborg, Sweden. }
\email{sjoza@chalmers.se}

\address{Ruadha\'i Dervan, Centre de math\'ematiques Laurent-Schwartz, \'Ecole Polytechnique, Cour Vaneau, 91120 Palaiseau, France}
\email{R.Dervan@dpmms.cam.ac.uk}


\begin{abstract}
In this paper we study K-polystability of arbitrary (possibly non-projective) compact K\"ahler manifolds 
admitting holomorphic vector fields. As a main result we show that existence of a constant scalar curvature K\"ahler (cscK) metric implies 'geodesic K-polystability', in a sense that is expected to be equivalent to K-polystability in general. In particular, 
in the spirit of an expectation of Chen-Tang \cite{ChenTang} 
we show that geodesic K-polystability implies algebraic K-polystability for polarized manifolds, so our main result recovers a possibly stronger version of results of Berman-Darvas-Lu \cite{BDL} in this case. 
As a key part of the proof we also study subgeodesic rays with singularity type prescribed by singular test configurations, and prove a result on asymptotics of the K-energy functional along such rays. 

In an appendix by R. Dervan it is finally deduced that geodesic K-polystability implies equivariant K-polystability. 
This improves upon the results of \cite{Dervanrelative} and proves that 
existence of a cscK (or extremal) K\"ahler metric implies equivariant K-polystability (resp. relative K-stability).
\end{abstract}

\maketitle

\section{Introduction}

\noindent This paper is concerned with the study of K-polystability of arbitrary (possibly non-projective) compact K\"ahler manifolds admitting holomorphic vector fields. This builds on the transcendental formalism for K-stability introduced in \cite{SD1}, which simultaneously simplifies and generalizes the classical theory for polarized manifolds (due to Donaldson \cite{Tian, Donaldsontoric} and others) to the setting of arbitrary K\"ahler manifolds. A strong motivation for studying K-stability is the Yau-Tian-Donaldson conjecture, which connects a classical question of Calabi \cite{Calabi50, Calabiextremal} on existence of "canonical metrics" to algebro-geometric stability notions.
Given a polarized manifold $(X,L)$ (i.e. a pair of a compact K\"ahler manifold $X$ and an ample line bundle $L$ on $X$) it predicts that the first Chern class $c_1(L)$ admits a constant scalar curvature K\"ahler (cscK) representative if and only if $(X,L)$ is K-stable.
For Fano manifolds this conjecture was recently confirmed in \cite{CDSone,CDStwo, CDSthree} and \cite{TianYTDconjecture}, see also \cite{LiTianWangYTDconjecturesingular} for an extension of this result to singular Fano varieties.
In the case of general polarized manifolds
it is known that K-polystability is a necessary condition for existence of cscK metrics in the first Chern class $c_1(L)$, see \cite{Berman, BDL}. The converse is a central open problem in K\"ahler geometry.  

Now let $(X,\omega)$ be a compact K\"ahler manifold and denote by $\alpha := [\omega] \in H^{1,1}(X,\mathbb{R})$ its associated K\"ahler class. From a differential geometric perspective it is natural to ask if it is possible to introduce a generalized K-stability notion that characterizes existence of cscK metrics in the K\"ahler class $\alpha$, even if it is not necessarily of the form $c_1(L)$ for some ample line bundle $L$ over $X$. Such a theory of K-stability for "transcendental" cohomology classes (i.e $\alpha \in H^{1,1}(X,\mathbb{R})$ that are not necessarily in $H^2(X,\mathbb{Q})$) was introduced in \cite{SD1} and  \cite{DervanRoss, Dervanrelative}. 
As a first result, it was proven in \cite{SD1, DervanRoss} that cscK manifolds are always K-semistable, thus generalizing results of \cite{Tian, Donaldsoncalabi}. We also proved (uniform) K-stability whenever the automorphism group is discrete. 

The case of compact K\"ahler manifolds admitting holomorphic vector fields is however much more involved. 
In this direction, it was proven in \cite{Dervanrelative} that if the norm $||(\mathcal{X},\mathcal{A})||$ of a test configuration\footnote{In the generalized sense of \cite{SD1}.} is strictly positive, then $\mathrm{DF}(\mathcal{X},\mathcal{A}) > 0$. In order to establish the natural analog of K-polystability in the general K\"ahler setting it therefore remains to characterize the vanishing of the norm and the Donaldson-Futaki invariant respectively. This is expected to happen precisely when the test configuration is a "product" (in the appropriate sense). 
In this paper we first give a definition of K-polystability for K\"ahler manifolds (Definition \ref{Def Kps}) and compare it to the usual notion for polarized manifolds. It turns out that our notion then \emph{implies} the classical one due to Donaldson \cite{Donaldsontoric}. As a main result we then establish a number of conditions equivalent to the vanishing of the Donaldson-Futaki invariant. These are of a more analytic nature, involving geodesic rays associated to test configurations (in the spirit of \cite{ArezzoTian, ChenTang}). This leads to a "geodesic K-polystability" notion defined by asking that $\mathrm{DF}(\mathcal{X},\mathcal{A})$ vanishes precisely if $(\mathcal{X},\mathcal{A})$ is "geodesically product", meaning that the geodesic ray associated to $(\mathcal{X},\mathcal{A})$ is induced by the flow of a real holomorphic Hamiltonian vector field (see Definition \ref{Definition geodesic Kps}).
We then prove the following: 

\begin{thm} \label{Main geodesic Kps thm} Suppose that $(X,\omega)$ is a compact K\"ahler manifold. If the K\"ahler class $\alpha := [\omega] \in H^{1,1}(X,\mathbb{R})$ admits a cscK representative, then $(X,\alpha)$ is geodesically K-polystable.
\end{thm}


\noindent This stability notion is natural in light of the classical circle of ideas relating test configurations to geodesic rays (see e.g. \cite{Chen00, ArezzoTian, ChenTang, WR, PRS, BHJ2}). In particular, in light of influential examples of \cite{ACGTF} it was suggested in e.g. \cite{ChenTang} that some kind of "geodesic stability" should be an interesting candidate to study alongside the classical algebraic stability. The expectation of \cite{ChenTang} was that such a notion should be stronger than algebraic stability. The following Proposition \ref{Prop equivalence} confirms the analogous result with respect to the geodesic K-polystability introduced in this paper. 
In case the automorphism group of $X$ is discrete, we moreover prove that geodesic K-polystability is equivalent to K-polystability.

\begin{prop} \label{Prop equivalence}
Let $(X,\omega)$ be a compact K\"ahler manifold with $\alpha := [\omega] \in H^{1,1}(X,\mathbb{R})$ the associated K\"ahler class. Then the following holds: 
\begin{itemize}
\item If $\alpha = c_1(L)$ for some ample line bundle $L$ over $X$ and $(X,c_1(L))$ is geodesically K-polystable, then $(X,L)$ is algebraically K-polystable in the classical sense of Donaldson \cite{Donaldsontoric}.
\item If $\mathrm{Aut}(X)$ is discrete, then $(X,\alpha)$ is geodesically K-polystable if and only if $(X,\alpha)$ is K-polystable in the sense of Definition \ref{Def Kps}. 
\end{itemize}
\end{prop}

\begin{rem} It follows from the first point that if we restrict to the special case of polarized manifolds, then we in particular recover the K-polystability results of \cite{Berman, BDL}. 
\end{rem}

\noindent 
The main Theorem \ref{Main geodesic Kps thm} is a consequence of a more elaborate analytic characterization of the vanishing of the Donaldson-Futaki invariant, of which the above is a special case. Its statement involves the  Monge-Amp\`ere energy functional $\mathrm{E}$, the Aubin $\mathrm{J}$-functional, and the intersection number $$\mathrm{M}^{\mathrm{NA}}(\mathcal{X},\mathcal{A}) := \mathrm{DF}(\mathcal{X},\mathcal{A}) + ((\mathcal{X}_{0,red} - \mathcal{X}_0) \cdot \mathcal{A}^n),$$
is a modification of the Donaldson-Futaki invariant, so that $\mathrm{M}^{\mathrm{NA}}(\mathcal{X},\mathcal{A}) \leq  \mathrm{DF}(\mathcal{X},\mathcal{A})$ with equality if and only the central fiber $\mathcal{X}_0$ is reduced. In case the test configuration $(\mathcal{X},\mathcal{A})$ has non-smooth total space $\mathcal{X}$ this quantity can be computed on any resolution (and it is independent of this choice). We refer to Sections \ref{Section preliminaries} and \ref{Section analytic characterization} for full definitions.

\begin{thm} \label{Main theorem analytic product conditions}
Suppose that $(X,\omega)$ is a cscK manifold, with $\alpha := [\omega] \in H^{1,1}(X,\mathbb{R})$ the corresponding K\"ahler class. Let $(\mathcal{X},\mathcal{A})$ be a normal and relatively K\"ahler test configuration for $(X,\alpha)$ whose associated geodesic ray $(\varphi_t)_{t \geq 0}$ satisfies $\mathrm{E}(\varphi_t) = 0$ for each $t \in [0, +\infty)$. Let $J: TX \rightarrow TX$ be the complex structure and $\omega$ a cscK metric on $X$. Then the following statements are equivalent:
\begin{enumerate}
\item $\mathrm{DF}(\mathcal{X},\mathcal{A}) = 0$. 
\item The central fiber $\mathcal{X}_0$ is reduced and $\mathrm{M}^{\mathrm{NA}}(\mathcal{X},\mathcal{A}) = 0$. 
\item The central fiber $\mathcal{X}_0$ is reduced and the Mabuchi K-energy functional is constant along the geodesic ray $(\varphi_t)_{t \geq 0}$ associated to $(\mathcal{X},\mathcal{A})$, i.e. we have $\mathrm{M}(\varphi_t) = \mathrm{M}(\varphi_0)$ for each $t \in [0,+\infty)$. 
\item The central fiber $\mathcal{X}_0$ is reduced and the associated geodesic ray satisfies
$$
\inf_{g \in G} \mathrm{J}(g. \varphi_t) = 0 \;\; \mathrm{and} \; \; \inf_{g \in G} d_1(0, g.\varphi_t)  = 0.
$$
\item The central fiber $\mathcal{X}_0$ is reduced and there is a real holomorphic Hamiltonian vector field $V$ such that the geodesic ray $(\varphi_t)_{t \geq 0}$ associated to $(\mathcal{X},\mathcal{A})$ satisfies $\mathrm{exp}(tV)^*\omega = \omega$ and $\mathrm{exp}(tJV)^*\omega = \omega_{\varphi_t}$.
\item The central fiber $\mathcal{X}_0$ is reduced and the associated geodesic ray $(\varphi_t)$ consists entirely of cscK potentials. More precisely, if $\bar{\mathcal{S}}$ denotes the mean scalar curvature of $\omega_{\varphi_0}$, then
$$
\mathrm{\mathcal{S}(\omega_{\varphi_t}) = \bar{\mathcal{S}}}
$$
for each $t \in [0, +\infty)$. 
\end{enumerate} 
\end{thm}

\noindent 
The central part of the proof builds on a study of energy functional asymptotics along subgeodesic rays with singularity type prescribed by a given test configuration.
This extends results from \cite{SD1} to the setting of singular test configurations, as explained below. 

\smallskip

\subsection{The singularity type class and asymptotics of the K-energy}

The point of view taken in this paper is to exploit and develop the relation between test configurations and geodesic rays, which is a key idea going back to \cite{Chen00, ArezzoTian, ChenspaceIII} and others. It is well known that, by solving a certain homogeneous complex Monge-Amp\`ere equation \cite{Semmes, Chen00, Donaldsongeodesicequation}, one may associate to a given test configuration a (unique, up to certain choices) geodesic ray in the space of K\"ahler potentials. The corresponding result in the K\"ahler case was proven in \cite{SD1, DervanRoss}. 
Working with the "associated geodesic ray" however turns out to be insufficient for the technique of the proof of Theorem \ref{Main theorem analytic product conditions}. Given a test configuration $(\mathcal{X},\mathcal{A})$ we instead consider a class of \emph{sub}geodesic rays with singularity type prescribed by $(\mathcal{X},\mathcal{A})$,
in the sense described in Section \ref{Subsection associated rays}. We then say that the subgeodesic ray is \emph{$L^{\infty}$-compatible} with $(\mathcal{X},\mathcal{A})$. We will also make use of a $\mathcal{C}^{\infty}$-compatibility condition, see Definition \ref{Definition compatibility smooth}.  
The notion that we introduce is tailored so that it behaves well with respect to energy functional asymptotics in K\"ahler geometry.  
In particular, this leads to the following more general version of \cite[Theorem 5.1]{SD1}, valid also for test configurations whose total space is not smooth:  

\begin{thm} \label{Theorem generalized thm C intro}
Let $(\mathcal{X},\mathcal{A})$ be a normal and relatively K\"ahler test configuration for $(X,\alpha)$. Let $\varphi_0 \in \mathcal{H}_0$. Then there is a smooth ray $[0,+\infty[ \ni t \mapsto \psi_t \in \mathrm{PSH}(X,\omega) \cap C^{\infty}(X) \cap E^{-1}(0)$ on $X$ emanating from $\varphi_0$, that is $C^{\infty}$-compatible with $(\mathcal{X},\mathcal{A})$ and satisfies
$$
\lim_{t \rightarrow +\infty} \frac{{\mathrm{M}}(\psi_t)}{t} = {{\mathrm{M}}^{\mathrm{NA}}}(\mathcal{X},\mathcal{A}).
$$
\end{thm}

\begin{rem} In case $\mathcal{X}$ is smooth and dominating (see Definition \ref{Definition smooth and dom}) the above result holds also for the unique weak geodesic ray associated to $(\mathcal{X},\mathcal{A})$, see \cite{SD1}. 
\end{rem}

\begin{rem} \emph{(Deligne functionals, \cite{SD1})} The step from smooth test configurations to allowing the central fiber to be singular is straightforward for many energy functionals in K\"ahler geometry (in particular this holds for the so called \emph{Deligne functionals} introduced in \cite{SD1}, see also Section \ref{Section Deligne functionals}). In case $\mathcal{X}$ is smooth the above result also holds for the associated geodesic ray, but this is not clear if the test configuration has singular total space. Interestingly, it seems convenient to consider the singularity type prescribed by the test configuration rather than the singularity type of the associated geodesic ray. 
\end{rem}

\noindent Given a test configuration $(\mathcal{X},\mathcal{A})$ for $(X,\alpha)$ we can therefore associate to it a set of subgeodesic rays on $X$ of prescribed singularity type, so that they are all $L^{\infty}$-compatible with $(\mathcal{X},\mathcal{A})$. This in turn gives an equivalence relation of test configurations by saying that two test configurations are equivalent if and only if there is a subgeodesic ray that is compatible with both of them. The reader may compare with the discussion on "parallell rays" introduced in \cite{ChenTang}. 

\noindent An advantage of the present approach is that the above equivalence relation on test configurations in fact determines the total space up to $\mathbb{C}^*$-equivariant isomorphism. Indeed, we have the following result:

\begin{thm} \label{Main lemma unique model intro}
Any subgeodesic ray $[0,+\infty) \ni t \mapsto \varphi_t \in \mathrm{PSH}(X,\omega) \cap L^{\infty}(X)$ is $L^{\infty}$-compatible with at most one normal, relatively K\"ahler test configuration $(\mathcal{X}, \mathcal{A})$ for $(X,\alpha)$. 
\end{thm}

\noindent More precisely, suppose that $(\mathcal{X},\mathcal{A})$ and $(\mathcal{X}',\mathcal{A}')$ are two cohomological test configurations for $(X,\alpha)$ whose associated geodesic rays are of the same singularity type. Then the natural $\mathbb{C}^*$-equivariant isomorphism $\mathcal{X} \setminus \mathcal{X}_0 \rightarrow \mathcal{X}' \setminus \mathcal{X}_0'$ factoring through $X \times \mathbb{P}^1 \setminus \{0\}$, extends to a $\mathbb{C}^*$-equivariant isomorphism $\mathcal{X} \rightarrow \mathcal{X}'$. The above can then be seen as an injectivity result for the assignment 
$$
(\mathcal{X},\mathcal{A}) \mapsto \left\{ (\varphi_t)_{t \geq 0}^{(\mathcal{X},\mathcal{A})} \right\}, 
$$
sending a test configuration to the equivalence class of subgeodesic rays that are $L^{\infty}$-compatible with it. 
This is however \emph{not} a bijection in general, since test configurations with the same total space (but different polarizations) have different associated singularity types. This result extends (with a completely different proof) the concept of \emph{unique ample model} in the context of polarized manifolds, see \cite{BHJ1}. 

\subsection{The case of $\mathrm{Aut}(X)$ discrete and applications to relative K-stability} 
We finally make explicit the consequences of the above results in the important special case of compact K\"ahler manifolds with discrete automorphism group. 
In particular, we then obtain the following strengthened version of Theorem \ref{Main theorem analytic product conditions}: 

\begin{thm} \label{Theorem equivalent characterizations discrete intro}
Suppose that $(X,\omega)$ is a cscK manifold with discrete automorphism group. Let $\alpha := [\omega] \in H^{1,1}(X,\mathbb{R})$ be the corresponding K\"ahler class. For any cohomological test configuration $(\mathcal{X},\mathcal{A})$ for $(X,\alpha)$ the following are equivalent: 
\begin{enumerate}
\item $\mathrm{DF}(\mathcal{X},\mathcal{A}) = 0$
\item $\mathrm{J}^{\mathrm{NA}}(\mathcal{X},\mathcal{A}) = 0$
\item $(\mathcal{X},\mathcal{A})$ is the trivial test configuration. 
\end{enumerate}
\end{thm}

\noindent It is worth pointing out that the above results entail that uniform K-stability implies K-stability, as the terminology suggests: 

\begin{cor}
If $(X,\alpha)$ is uniformly K-stable then it is K-stable. 
\end{cor}

\noindent This gives one way of showing that cscK manifolds are K-stable. In fact, the part of proving uniform K-stability using the energy functional asymptotics developed in \cite{SD1, DervanRoss} is not much more difficult than proving K-semistability. 
In this setting we can also prove that geodesic K-stability is equivalent to K-stability in the usual sense. This then yields a new and direct proof of K-stability of cscK manifolds (cf. \cite{DervanRoss} and \cite{SD1} for previously known proofs): 

\begin{cor}
Suppose that $(X,\omega)$ is a cscK manifold admitting no non-trivial holomorphic vector fields. Then $(X,\alpha)$ is K-stable. 
\end{cor}

\noindent The idea of the proof is as follows: Suppose that $\alpha \in H^{1,1}(X,\mathbb{R})$ admits a cscK representative. By the K-semistability result \cite[Theorem A]{SD1} (see also \cite[Theorem 1.1]{DervanRoss})
it suffices to consider test configurations $(\mathcal{X},\mathcal{A})$ for $(X,\alpha)$ that satisfy $\mathrm{DF}(\mathcal{X},\mathcal{A}) = 0$. We then wish to show that $(\mathcal{X},\mathcal{A})$ is trivial. The strategy of the proof is to show that the geodesic ray associated to such a test configuration must be the constant ray. 
As a particular case of the injectivity result Theorem \ref{Main lemma unique model intro} we then obtain the sought result. 


In an appendix by R. Dervan we finally point out implications of the main results of this paper also for equivariant K-polystability and relative K-stability for K\"ahler manifolds, studied in \cite{Dervanrelative}. First of all, in order to further compare the main results of this paper with the existing literature, we relate geodesic K-polystability also to equivariant K-polystability. It turns out that geodesic K-polystability is the stronger notion:  

\begin{thm} \label{Appendix intro 1}
If $(X,\alpha)$ is geodesically K-polystable, then it is equivariantly K-polystable \emph{(in the sense of Definition \ref{Appendix def equivariant stability})}. 
\end{thm}

\noindent This result improves upon \cite{Dervanrelative} and moreover yields the following immediate corollary: 

\begin{cor}
\label{Appendix main result} If $(X,\alpha)$ admits a cscK metric, then it is equivariantly K-polystable.  
\end{cor}

\noindent Relying on Theorem \ref{Theorem equivalent characterizations discrete intro} combined with the injectivity result \ref{Main lemma unique model intro}, as well as an idea from \cite{cs} and \cite{Dervanrelative}, the above result can also be given a direct proof. These ideas can also be immediately applied also to the case of relative K-stability of extremal K\"ahler manifolds (see Definition \ref{Appendix def equivariant stability}), thus improving upon the main result of \cite{Dervanrelative}:   

\begin{thm} \label{Appendix intro 2}
If $(X,\alpha)$ admits an extremal K\"ahler metric, then it is relatively K-stable. 
\end{thm}


\subsection{Organization of the paper}
In Section \ref{Section preliminaries} we recall preliminaries on energy functional asymptotics and the formalism of K-stability for K\"ahler manifolds with transcendental cohomology class that was introduced in \cite{SD1}. In particular, we give the definition of K-polystability for K\"ahler manifolds (Definition \ref{Def Kps}). 
In Section \ref{Section energy functional asymptotics} we study subgeodesic rays with singularity type prescribed by test configurations and give results for the asymptotics of energy functionals (including the K-energy) along such rays. We moreover prove that the equivalence class of subgeodesic rays associated to a given test configuration determines the total space up to $\mathbb{C}^*$-equivariant isomorphism. In particular, we prove Theorems \ref{Theorem generalized thm C intro} and \ref{Main lemma unique model intro}. 
In Section \ref{Section analytic characterization} we finally apply the above to yield analytic characterizations of test configurations with vanishing Donaldson-Futaki invariants. This leads to a natural notion of geodesic K-polystability, as discussed in Section \ref{Section geodesic kps}. Theorem \ref{Main theorem analytic product conditions} is proven in Section \ref{Section analytic characterization}. For the main Theorem \ref{Main geodesic Kps thm} and Proposition \ref{Prop equivalence}, see Section \ref{Section geodesic kps}. Theorems \ref{Appendix intro 1} and \ref{Appendix intro 2} are proven in the appendix.


\section{Preliminaries} \label{Section preliminaries}


\noindent Let $(X, \omega)$ be a compact K\"ahler manifold of complex dimension $n$ and denote by $\alpha := [\omega] \in H^{1,1}(X,\mathbb{R})$ the associated K\"ahler class. Denote by $\ric{\omega} = -dd^c\log \omega^n$ the Ricci curvature form, where $dd^c := \frac{i}{2\pi}\partial\bar{\partial}$ is normalised so that $\ric{\omega}$ represents the first Chern class $c_1(X) := c_1(-K_X)$. 
The scalar curvature of $\omega$ is then given by the trace
\begin{equation*}
S(\omega) := \mathrm{Tr}_{\omega} \mathrm{Ric}(\omega) =  n \frac{\ric{\omega} \wedge \omega^{n-1}}{\omega^n}
\end{equation*} 
of the Ricci curvature form. Note that $S(\omega)$ is nothing but a smooth function on $X$. If $S(\omega)$ is constant (equal to the average scalar curvature $\bar{\mathcal{S}}$) then $\omega$ is said to be a constant scalar curvature K\"ahler (cscK) metric. We then say that $(X,\omega)$ is a \emph{cscK manifold}. A K\"ahler class $\beta \in H^{1,1}(X,\mathbb{R})$ is said to \emph{admit a cscK representative} if there is a K\"ahler form $\eta$ such that $\beta = [\eta]$ and $\mathcal{S}(\eta) = \bar{\mathcal{S}}$. Note finally that 
the average scalar curvature is the cohomological constant given by
\begin{equation}
\bar{\mathcal{S}} := V^{-1} \int_X S(\omega) \; \omega^n =  n\frac{\int_X c_1(X) \cdot \alpha^{n-1}}{ \int_X \alpha^n} := n  \frac{(c_1(X) \cdot  \alpha^{n-1})_X}{( \alpha^n)_X},
\end{equation}
where $V := \int_X \omega^n := ( \alpha^n)_X  $ is the K\"ahler volume.

\medskip

\subsection{Energy functionals in K\"ahler geometry} \label{section energy functionals}

We here briefly recall the variational setup of \cite{BBGZ}. 
In particular, we fix our notation for quasi-plurisubharmonic functions and energy functionals in K\"ahler geometry. We also recall the definition of \emph{Deligne functionals} introduced in \cite{SD1}. 

\subsubsection{The space of K\"ahler metrics} \label{Section space} Let $\theta$ be a closed $(1,1)$-form on $X$ and denote, as usual, by $\mathrm{PSH}(X,\theta)$ the space of $\theta$-psh functions $\varphi$ on $X$, i.e. the set of functions that can be locally written as the sum of a smooth and a plurisubharmonic function, and such that $\theta_{\varphi} := \theta + dd^c\varphi \geq 0 $
in the weak sense of currents. In particular, if $\omega$ is our fixed K\"ahler form on $X$, then we write 
$$
\mathcal{H}_{\omega}  := \{ \varphi \in \mathcal{C}^{\infty}(X) : \omega_{\varphi} := \omega + dd^c\varphi > 0 \} \subset \mathrm{PSH}(X,\omega)
$$ 
for the space of smooth K\"ahler potentials on $X$. This space was first studied by Mabuchi in \cite{Mabuchisymplectic}. The space of K\"ahler metrics representing $\alpha$ is then given by
$
\mathcal{K} = \mathcal{H}/\mathbb{R},
$
since potentials are defined up to constants. It is an easy consequence of the definition that $\mathcal{H}$ is an \emph{open} and \emph{convex} subset of $C^{\infty}(X)$. Moreover, for each $\varphi \in C^{\infty}(X)$ we have $T_{\varphi}(\mathcal{H}) = C^{\infty}(X)$. 
As first introduced by Mabuchi \cite{Mabuchisymplectic} the spaces $\mathcal{H}$ and $\mathcal{K}$ can be endowed with a natural Riemannian $L^2$-metric that is often denoted $d_2$.  
Following Darvas \cite{DR}, \cite{Darvas, Darvas14, Darvas15} and others we however consider $\mathcal{H}$ as a path metric space endowed with a certain \emph{Finsler metric} $d_1$. To introduce it, let $d_1: \mathcal{H}_{\omega} \times \mathcal{H}_{\omega} \rightarrow \mathbb{R}_+$ be the path length pseudometric associated to the weak Finsler metric on $\mathcal{H}_{\omega}$ defined by
$$
||\xi||_{\varphi} := V^{-1} \int_X |\xi|\omega_{\varphi}^n, \; \; \; \xi \in T_{\varphi}\mathcal{H}_{\omega} = \mathcal{C}^{\infty}(X). 
$$
More explicitly, if $[0,1] \ni t \mapsto \phi_t$ is a smooth path in $X$, then let
$$
l_1(\phi_t):= \int_0^1 ||\dot{\phi}_t||_{\phi_t} dt
$$
be its length, and define 
$$
d_1(\varphi, \psi) = \inf \left\{ l_1(\phi_t), \; \; (\phi_t)_{0\leq t \leq 1} \subset \mathcal{H}_{\omega}, \; \phi_0 = \varphi, \; \phi_1 = \psi \right\},
$$
the infimum being taken over smooth paths $t \mapsto \phi_t$ as above. It can be seen (see \cite[Theorem 2]{Darvas}) that $(\mathcal{H}_{\omega}, d_1)$ is a metric space.  We refer to the survey article \cite{Darvassurvey} for a detailed treatment and background on this topic.

\subsubsection{Basic energy functionals} \label{multivariate energy functional}

We now introduce the notation for energy functionals that we will use. 
Now let $\varphi_i \in \mathrm{PSH}(X,\omega) \cap L^{\infty}_{\mathrm{loc}}$. 
If $\varphi_i \in \mathrm{PSH}(X,\theta) \cap L^{\infty}_{\mathrm{loc}}$, $1 \leq i \leq p \leq n$,  it follows from the work of Bedford-Taylor \cite{BT} that we can give meaning to the product $\bigwedge_{i = 1}^p (\theta + dd^c\varphi_i)$, which then defines a closed positive $(p,p)$-current on $X$. As usual, we then define the \emph{Monge-Amp\`ere measure} as the following probability measure, given by the top wedge product 
\begin{equation*}
\mathrm{MA}(\varphi) := V^{-1}(\omega + dd^c\varphi)^{n}.
\end{equation*} 
\noindent The \emph{Monge-Amp\`ere energy functional} (or \emph{Aubin-Mabuchi functional}) $\mathrm{E} := \mathrm{E}_{\omega}$ is then defined by 
\begin{equation*}
\mathrm{E}(\varphi) := \frac{1}{n+1}\sum_{j=0}^n V^{-1} \int_X \varphi (\omega + dd^c\varphi)^{n-j} \wedge \omega^j
\end{equation*}
(such that $\mathrm{E}' = \mathrm{MA}$ and $\mathrm{E}(0) = 0$). Similarily, if $\theta$ is any closed  $(1,1)$-form, we define a functional $\mathrm{E}^{\theta} := \mathrm{E}_{\omega}^{\theta}$ by
\begin{equation*}
\mathrm{E}^{\theta}(\varphi) :=  \sum_{j=0}^{n-1} V^{-1} \int_X \varphi (\omega + dd^c\varphi)^{n-j-1} \wedge \omega^j \wedge \theta,
\end{equation*}
and we will also have use for the Aubin $\mathrm{J}$-functional $\mathrm{J}: \mathrm{PSH}(X,\omega) \cap L^{\infty}_{\mathrm{loc}} \rightarrow \mathbb{R}_{\geq 0}$ defined by $$\mathrm{J}(\varphi) := V^{-1} \int_X \varphi \; \omega^n - \mathrm{E}(\varphi),$$ which essentially coincides with the \emph{minimum norm} of a test configuration (see \cite{Dervanuniform, BHJ1}). 
More generally, it is possible to define a natural \emph{multivariate} version of the Monge-Amp\`ere energy, of which all of the above functionals are special cases. 

\subsubsection{Deligne functionals} \label{Section Deligne functionals} Let $\theta_0, \dots, \theta_n$ be closed $(1,1)$-forms on $X$. Motivated by corresponding properties for the \emph{Deligne pairing} (cf. e.g. \cite{Berman}, \cite{Elkik} for background) we would like to consider a \emph{functional} $\langle \cdot, \dots, \cdot \rangle_{(\theta_0,\dots,\theta_n)}$ on the space $\mathrm{PSH}(X,\theta_0) \cap L^{\infty}_{\mathrm{loc}} \times \dots \times \mathrm{PSH}(X,\theta_n) \cap L^{\infty}_{\mathrm{loc}}$ 
that is
 \begin{itemize}
 \item symmetric, i.e. for any permutation $\sigma$ of the set $\{0,1, \dots,n\}$, we have
\begin{equation*}
\langle \varphi_{\sigma(0)}, \dots, \varphi_{\sigma(n)} \rangle_{(\theta_{\sigma(0)}, \dots, \theta_{\sigma(n)})} = \langle \varphi_0, \dots, \varphi_n \rangle_{(\theta_0,\dots,\theta_n)}. 
\end{equation*}  
 \item  if $\varphi_0'$ is another $\theta_i$-psh function in $\mathrm{PSH}(X,\theta) \cap L^{\infty}_{\mathrm{loc}}$, then we have a 'change of function' property
\begin{equation*}
\langle \varphi_0', \varphi_1 \dots, \varphi_n \rangle - \langle \varphi_0, \varphi_1 \dots, \varphi_n \rangle  = 
\end{equation*}
\begin{equation*}
 \int_X (\varphi_0' - \varphi_0) \; (\omega_1 + dd^c\varphi_1) \wedge \dots \wedge (\omega_n + dd^c\varphi_n). 
\end{equation*}
\end{itemize}  
Demanding that the above properties hold necessarily leads to the following definition of \emph{Deligne functionals}, that will provide a useful terminology for this paper: 

\begin{mydef} \label{Multivariate energy def} Let $\theta_0, \dots, \theta_n$ be closed $(1,1)$-forms on $X$. Define a \emph{multivariate energy functional} $\langle \cdot, \dots, \cdot \rangle_{(\theta_0,\dots,\theta_n)}$ on the space $\mathrm{PSH}(X,\theta_0) \cap L^{\infty}_{\mathrm{loc}} \times \dots \times \mathrm{PSH}(X,\theta_n) \cap L^{\infty}_{\mathrm{loc}}$ \emph{($n + 1$ times)} by 
\begin{equation*}
 \langle\varphi_0, \dots, \varphi_n \rangle_{(\theta_0,\dots,\theta_n)} :=  \int_X \varphi_0 \; (\theta_1 + dd^c\varphi_1) \wedge \dots \wedge (\theta_n + dd^c\varphi_n)
\end{equation*} 
\begin{equation*}
 + \int_X \varphi_1 \; \theta_0 \wedge (\theta_2 + dd^c\varphi_2) \wedge \dots \wedge (\theta + dd^c\varphi_n) +  \dots + \int_X \varphi_n \; \theta_0 \wedge \dots \wedge \theta_{n-1}.
\end{equation*}
\end{mydef} 


\begin{rem}The multivariate energy functional $\langle \cdot, \dots, \cdot \rangle_{(\theta_0,\dots,\theta_n)}$ can also be defined on $\mathcal{C}^{\infty}(X) \times \dots \times \mathcal{C}^{\infty}(X)$ by the same formula. 
\end{rem}

\begin{prop} \emph{(\cite[Proposition 2.3]{SD1})}
The functional $\langle \cdot, \dots, \cdot \rangle_{(\theta_0,\dots,\theta_n)}$ is symmetric. 
\end{prop}

\begin{ex} \label{Example Deligne pairing} The functionals $\mathrm{E}$, $\mathrm{E}^{\theta}$ and $\mathrm{J}$ can be written using the above multivariate energy functional formalism. Indeed, if $\theta$ is a closed $(1,1)$-form on $X$, $\omega$ is a K\"ahler form on $X$ and $\varphi$ is an $\omega$-psh function on $X$, then
\begin{equation*}
\mathrm{E}(\varphi) = \frac{1}{n+1} V^{-1} \langle \varphi, \dots, \varphi \rangle_{(\omega, \dots, \omega)} \;, \; \; \mathrm{E}^{\theta}(\varphi) =  V^{-1} \langle 0,\varphi, \dots, \varphi \rangle_{(\theta, \omega, \dots, \omega)}
\end{equation*} 
and 
\begin{equation*}
\mathrm{J}(\varphi)  =  V^{-1} \langle \varphi,0, \dots, 0 \rangle_{(\omega, \dots, \omega)} - \mathrm{E}(\varphi).
\end{equation*}
Compare also \cite[Example 5.6]{Rubinsteinbook} on Bott-Chern forms. 
\end{ex}

\subsubsection{The Mabuchi K-energy functional}

Let $\omega$ be a K\"ahler form on $X$ and consider any smooth path $(\varphi_t)_{t \geq 0}$ in the space $\mathcal{H}$ of K\"ahler potentials on $X$. The \emph{Mabuchi functional} (or \emph{K-energy} functional) $\mathrm{M}: \mathcal{H} \rightarrow \mathbb{R}$, first introduced in \cite{MabuchiKenergy1, Mabuchi}, is defined by its Euler-Lagrange equation
\begin{equation*} 
\frac{d}{dt} \mathrm{M}(\varphi_t) = - V^{-1} \int_X \dot{\varphi}_t(\mathcal{S}(\omega_{\varphi_t}) - \bar{\mathcal{S}})\; \omega_{\varphi_t}^n.
\end{equation*} 
The Mabuchi functional is independent of the path chosen and its critical points are precisely the cscK metrics, when they exist. 

There is also an explicit formula (cf. \cite{Chen}) for the K-energy functional in terms of an "energy" and an "entropy" part, in the following way: 

\begin{equation} \label{Chens formula} 
{\mathrm{M}}(\varphi) = \left(\bar{\mathcal{S}} \mathrm{E}(\varphi) - \mathrm{E}^{\ric{\omega}}(\varphi)\right) + V^{-1} \int_X \log \left( \frac{\omega_{\varphi}^n}{\omega^n} \right) \omega_{\varphi}^n,
\end{equation}
Here the latter term is the \emph{relative entropy} of the probability measure $\mu := \omega_{\varphi}^n/V$ with respect to the reference measure $\mu_0 := \omega^n/V$. Recall that the entropy takes values in $[0, +\infty]$ and is finite if $\mu/\mu_0$ is bounded. It can be seen to be always lower semi-continuous (lsc) in $\mu$.

Following Chen \cite{Chen} (using the formula \eqref{Chens formula}) we will often work with the extension of the Mabuchi functional to the space of $\omega$-psh functions with bounded Laplacian. This is a natural setting to consider, since the space $(\mathcal{H},d_2)$  of K\"ahler potentials endowed with the Mabuchi metric is \emph{not} geodesically complete, while weak geodesic rays with bounded Laplacian are known to always exist, cf. \cite{Chen00, Blocki13, Darvas14, DL12} and \cite{Tosattigeodesicregularity}.  

\subsubsection{Group actions} \label{Section G-action} Let $G := \mathrm{Aut}_0(X)$ be the connected component of the complex Lie group of automorphisms of $X$. We denote its Lie algebra by $\mathfrak{aut}(X)$. It consists of infinitesimal automorphisms composed of real vector fields $V$ satisfying $\mathcal{L}_V J = 0$. We now recall the definition of the action of $G$ on the space $\mathcal{H}_0 := \mathcal{H} \cap E^{-1}(0)$ of normalized K\"ahler potentials, following \cite[Section 5.2]{DR} as a reference.  

First recall that $\mathcal{H}_0$ is in one-one correspondence with the set $\mathcal{H} := \{\omega_{\varphi}:= \omega + dd^c\varphi: \varphi \in \mathcal{C}^{\infty}(X), \omega_{\varphi} > 0 \}$. The group $G$ acts on $\mathcal{H}$ by pullback, i.e.
$$
g \cdot \xi := g^*\xi, \; \; \; \; g \in G, \; \xi \in \mathcal{H}.
$$
Due to the one-one correspondence with $\mathcal{H}$, the group $G$ also acts on potentials in $\mathcal{H}_0$, so that $g \cdot \varphi$ is the unique element in $\mathcal{H}_0$ satisfying $g \cdot \omega_{\varphi} = \omega_{g \cdot \varphi}$. As in \cite[Lemma 5.8]{DR}) one may show that
\begin{equation} \label{equation action}
g \cdot \varphi =  g \cdot 0 + \varphi \circ g. 
\end{equation}
For future use, we emphasize that the function $g \cdot 0$ is smooth (hence bounded) on $X$, as follows from the $dd^c$-lemma since $g^*\omega$ is always a K\"ahler form cohomologous to $\omega$.

\begin{prop}
Let $(X,\omega)$ be a given compact K\"ahler manifold. Then for any $g \in G$, the function $g \cdot 0$ is smooth on $X$. 
\end{prop}

\subsubsection{G-coercivity}
Following \cite[Definition 2.1]{ZhouZhu} and \cite[Definition 2.5]{TianEinsteinFano} we consider a version of Aubin's $\mathrm{J}$-functional defined on the orbit space $\mathcal{H}/G$ by
$$
\mathrm{J}_G(G\varphi) := \inf_{g \in G} \mathrm{J}(g.\varphi), \; \; \varphi \in \mathcal{H}_0.
$$
\begin{mydef}
We say that ${\mathrm{M}}$ is \emph{G-coercive} if there are constants $\delta, C > 0$ such that 
$$
{\mathrm{M}}(\varphi) \geq \delta \mathrm{J}_G(G\varphi) - C
$$ 
for all $\varphi \in \mathcal{H}_0$.
\end{mydef} 

\begin{thm}\emph{(\cite{BDL})} Suppose that $(X,\omega)$ is a K\"ahler manifold. If $\alpha := [\omega] \in H^{1,1}(X,\mathbb{R})$ admits a constant scalar curvature representative, then ${\mathrm{M}}$ is G-coercive.
\end{thm}

\noindent This deep result is a version of Tian's properness conjecture. 

\subsubsection{The $\mathrm{J}_G$-functional and the Finsler metric $d_{1,G}$} 

Now consider the path metric space $(\mathcal{H},d_1)$ (see Section \ref{Section space}). 
The action of $G$ on $\mathcal{H}_0$ is then a $d_1$-isometry \cite[Section 5]{DR}. It thus induces a pseudometric $d_{1,G}: \mathcal{H}/G \rightarrow \mathbb{R}_+$ on the orbit space, given by
$$
d_{1,G}(G\varphi,G\psi) := \inf_{f,g \in G} d_1(f\cdot \varphi, g \cdot \psi) 
$$
Note that $d_{1,G}(G0,G\varphi) = \inf_{g \in G} d_1(0,g \cdot \varphi)$. It is crucial for this paper to recall that the $\mathrm{J}_G$-functional is comparable to $d_{1,G}$. 
\begin{lem} \emph{(\cite[Lemma 5.11]{DR})} If $\varphi \in \mathcal{H}_0$, then 
$$
\frac{1}{C} \mathrm{J}_G (G\varphi) - C \leq d_{1,G}(G0,G\varphi) \leq C \mathrm{J}_G(G\varphi) + C, 
$$
for some $C > 0$. 
\end{lem}

\noindent See also \cite[Proposition 5.5]{DR}. For simplicity of notation we will in the sequel write simply $J_G(\varphi)$ instead of $J_G(G\varphi)$, and similarily for $d_{1,G}$.

\medskip

\subsection{Background on holomorphic vector fields, flows and the Futaki invariant} \label{Section holomorphic vector fields and projection} 

Let $(X,\omega)$  be a compact K\"ahler manifold and denote by $J: TX \rightarrow TX$ the associated complex structure. We take here the viewpoint of considering \emph{real} holomorphic vector fields on $X$. We denote the real tangent bundle of $X$ by $TX$. By a real vector field we mean a section of $TX$. On the other hand we may consider $TX^{1,0}$, i.e. the eigensubbundle of $TX \otimes \mathbb{C}$ corresponding to the eigenvalue $i$ of the extension of $J: TX \rightarrow TX$ to $TX \otimes \mathbb{C}$. We then have an identification as follows: 
\begin{equation} \label{eq identification}
TX \longrightarrow TX^{1,0}
\end{equation}
$$
V \mapsto V^{\mathbb{C}} := \frac{1}{2}(V - iJV).
$$ 


\begin{mydef}
A real vector field  $V$ on $X$ is said to be \emph{real holomorphic} if the flow preserves the complex structure, i.e. it has vanishing Lie derivative $L_VJ = 0$. Equivalently, a real vector field $V$ on $X$ is \emph{holomorphic} if the corresponding $(1,0)$-field $V^{\mathbb{C}}$ is a holomorphic section of the bundle $T^{1,0}X$. 
\end{mydef}

\noindent Recall that a holomorphic vector field on a \emph{compact} manifold is automatically $\mathbb{C}$-complete,
and then its \emph{flow} $\phi_t$ is an action of $(\mathbb{C},+)$ on $X$ by holomorphic automorphisms. Conversely, one may associate to every additive action $\phi: \mathbb{C} \times X \rightarrow X$ by holomorphic automorphisms on $X$ the vector field 
$$
V_{\phi}(x) := \frac{d}{dt}  \phi(t,x)_{\vert t = 0},
$$
called the \emph{inifinitesimal generator} of $X$. It is holomorphic and $\mathbb{C}$-complete on $X$, with the flow $\phi$. In the sequel we privilege the point of view of working with real holomorphic vector fields $V$, keeping in mind the identification \eqref{eq identification}.

\begin{mydef}
A real holomorphic vector field $V$ on $X$ is said to be \emph{Hamiltonian} if it admits a \emph{Hamiltonian potential} $h_{\omega}^{V} \in C^{\infty}(X,\mathbb{R})$ such that the contraction 
$$
i_{V}(\omega) := V \rfloor \omega = \sqrt{-1} \bar{\partial} h_{\omega}^V .
$$
\end{mydef}

\begin{rem}
Equivalently, a real holomorphic vector field admits a Hamiltonian potential if and only if it has a zero somewhere, see LeBrun-Simanca \cite{LeBrunSimanca}. 
\end{rem}  

\noindent The Hamiltonian potential is unique up to constants. To relieve this ambiguity we impose the normalization
$$
\int_X h_{\omega}^V \omega^n = 0.
$$

\begin{rem} For the purpose of comparing with the situation for polarized manifolds $(X,L)$ it is interesting to recall that Hamiltonian vector fields are precisely those that lift to line bundles, see \cite[Lemma 12]{Donaldsontoric}. This is a key property used e.g. in the proof \cite{BDL} of K-polystability for polarized cscK manifolds, due to R. Berman, T. Darvas. C. Lu. 
\end{rem}

\noindent 
Note that a real holomorphic Hamiltonian vector field is automatically a \emph{Killing field}, since $L_V J = L_V \omega = 0$ implies that also $L_V g = 0$ for the Riemannian metric associated to the K\"ahler form $\omega$. We further recall that since $X$ is compact it follows that the Lie algebra of real holomorphic vector fields on $X$ is finite dimensional. It is the (complex) Lie algebra of the complex Lie group $\mathrm{Aut}(X)$. We refer to e.g the expository notes \cite{Gabornotes} for details and further references.

We further recall the definition of the Futaki invariant, originally introduced in \cite{Futaki}. 

\begin{mydef}
Let $(X,\omega)$ be a compact K\"ahler manifold with $\alpha := [\omega] \in H^{1,1}(X,\mathbb{R})$ the corresponding K\"ahler class. Let $\mathfrak{h}$ be the Lie algebra of real holomorphic Hamiltonian vector fields on $X$. The Futaki invariant is the Lie algebra character $\mathrm{Fut}_{\alpha}: \mathfrak{h} \rightarrow \mathbb{R}$ defined by 
$$
\mathrm{Fut}_{\alpha}(V) := \int_X h_{\omega}^V(S(\omega) - \bar{S}) \; \omega^n,
$$
where $V$ be a real holomorphic Hamiltonian vector field on $X$ and $h_{\omega}^V \in C^{\infty}(X,\mathbb{R})$ is the associated Hamiltonian.
\end{mydef}

\noindent Note that the Futaki invariant is independent of the K\"ahler class considered: 

\begin{thm} \emph{(\cite{Futaki})} 
The above character is independent of representative of the K\"ahler class $\alpha := [\omega]$. In particular, if $[\omega]$ admits a cscK metric, then $\mathrm{Fut}_{\alpha}(V) = 0$ for each $V \in \mathfrak{h}$. \end{thm}

\begin{rem} There are several other common situations in which the Futaki character $\mathrm{Fut}_{\alpha}$ vanishes identically: In particular, it suffices that $(X,\alpha)$ K-semistable, as explained in the proof of Proposition \ref{Prop product vanish}. More generally, if the connected Lie group $\mathrm{Aut}_0(X)$ of automorphisms of $X$ is semisimple, then $\mathrm{Fut}_{\alpha}(V) = 0$ for each $V \in \mathfrak{h}$ and for \emph{each} K\"ahler class $\alpha$ on $X$, see \cite{LeBrunSimanca}. 
\end{rem}

\medskip

\subsection{K-polystability for K\"ahler manifolds}

In this section we recall the transcendental K-stability notions introduced in \cite{SD1}. In particular, we define K-polystability for K\"ahler manifolds and compare it to the classical algebraic notion for polarized manifolds. 

\subsubsection{Complex spaces and K\"ahler geometry}
In order to define test configurations for K\"ahler manifolds we will be concerned with 
singular (normal) complex analytic spaces, that we shall write in \emph{cursive}, such as $\mathcal{X}$ or $\mathcal{Y}$. If $\mathcal{X}$ is a complex space we denote respectively by $\mathcal{X}_{\mathrm{reg}}$ and $\mathcal{X}_{\mathrm{sing}}$ the regular and singular locus of $\mathcal{X}$. 
In this paper we restrict attention to \emph{normal} complex spaces. We refer to \cite{Demaillycomplexspaces} or \cite{Demaillybook} for a detailed presentation of these concepts. 

The notions of K\"ahler form and plurisubharmonic (psh) function can be defined on normal complex spaces. We briefly recall the definitions: Let $j: \mathcal{X} \rightarrow \Omega$ be an embedding of $\mathcal{X}$ into an open subset $\Omega \subset \mathbb{C}^N$. A $(p,q)$-form on $\mathcal{X}$ is then defined as the image of the restriction map 
$$
j^*: \mathcal{A}^{p,q}(\Omega) \rightarrow \mathcal{A}^{p,q}(\mathcal{X}). 
$$
It can be checked that this notion is well-defined (see \cite{Demaillycomplexspaces}). Indeed, if $j': \mathcal{X} \rightarrow \Omega' \subset \mathbb{C}^{N'}$ is another embedding then there exist (locally) holomorphic functions $f: \Omega \rightarrow \mathbb{C}^{N'}$ and $g: \Omega' \rightarrow \mathbb{C}^N$ so that $j' = fj$ and $j = gj'$, and it can then be seen that the respective images of $j^*$ and $j'^*$ coincide. In particular, a form is smooth if it is given locally by restriction of a smooth form under an embedding of $\mathcal{X}$ into an open subset of $\mathbb{C}^N$. Likewise, we say that a  smooth $(1,1)$-form $\omega$ is a \emph{K\"ahler form} on $\mathcal{X}$ if it is locally the restriction of a smooth K\"ahler form under an embedding of $\mathcal{X}$ as above. 
The notion of currents can be defined by means of duality as in the smooth case, 
see \cite{Demaillycomplexspaces}. 

Moreover, a function $f: \mathcal{X} \rightarrow [-\infty, +\infty[$ on the normal analytic space $\mathcal{X}$ is said to be \emph{plurisubharmonic}\footnote{There are several notions of plurisubharmonicity on complex spaces. They all coincide in case the space is locally irreducible. } (psh)  if it is upper semi-continuous (usc) on $\mathcal{X}$, not locally $-\infty$, and extends to a psh function in some local embedding of $\mathcal{X}$ into $\mathbb{C}^N$. 
A continuous function is psh iff its restriction to $\mathcal{X}_{\mathrm{reg}}$ is so \cite{FN80}. Further recall that bounded
psh functions on $\mathcal{X}_{\mathrm{reg}}$ extends to $\mathcal{X}$. We can use the notion of psh functions in order to view K\"ahler metrics (and currents) as equivalence classes of potentials as in \cite[p. 413]{GZ}: A K\"ahler potential on $\mathcal{X}$ is a family $(U_i,\varphi_i)_{i \in I}$ where $(U_i)_{i \in I}$ is an open covering of $\mathcal{X}$ and $\varphi_i$ a $C^{\infty}$-smooth strictly psh function such that $\varphi_i - \varphi_j$ is pluriharmonic on $U_i \cap U_j$. We say that two potentials $(U_i, \varphi_i)$ and $(V_j,\psi_j)$ are \emph{equivalent} iff $\varphi_i - \psi_j$ is pluriharmonic on $U_i \cap V_j$. A \emph{K\"ahler metric} is an equivalence class of K\"ahler potentials. A \emph{positive current} on $\mathcal{X}$ is an equivalence class of plursubharmonic potentials (defined analogously to K\"ahler potentials, in the obvious way). 
If a positive current $T$ has locally bounded potentials, it is fully
determined by the closed $(1,1)$-current $T_{\mathrm{reg}}$ on $\mathcal{X}_{\mathrm{reg}}$ defined on $U_i$ by
$T_{\mathrm{reg}} = dd^c\varphi_i$.
The notion of quasi-psh function is defined as in the smooth case: Let $\Omega$ be a K\"ahler metric on $\mathcal{X}$ with potential $(U_i, \varphi_i)$. An upper semi-continuous function $\varphi: \mathcal{X} \rightarrow [-\infty, +\infty[$ is $\Omega$-psh iff $\varphi_i + \varphi$ is psh on $U_i$, for all $i$. We denote by $\Omega + dd^c\varphi$ the positive $(1,1)$-current whose potential is $(U_i, \varphi_i + \varphi)$. 

We also recall that the Monge-Amp\`ere operator $(dd^c\psi)^n$ can be defined on a complex space, via e.g. Bedford-Taylor's study of Monge-Amp\`ere measures for locally bounded psh functions. The definition uses the fact that these measures do not charge proper analytic subsets, as explained in e.g. \cite[Proposition 16.42]{GZ}. 

\subsubsection{Cohomological test configurations} \label{section test config} 
We now recall the definition of \emph{cohomological test configurations} for K\"ahler manifolds.
 As a reference for this section we use \cite[Section 3]{SD1}. 

\begin{mydef} A \emph{cohomological test configuration $(\mathcal{X}, \mathcal{A}) := (\mathcal{X}, \mathcal{A}, \pi, \rho)$ for $(X,\alpha)$} consists of  
\begin{itemize}
\item a normal complex space $\mathcal{X}$, compact and K\"ahler, with a flat 
morphism $\pi: \mathcal{X} \rightarrow \mathbb{P}^1$
\item a $\mathbb{C}^*$-action $\rho$ on $\mathcal{X}$ lifting the canonical action on $\mathbb{P}^1$
\item a $\mathbb{C}^*$-equivariant isomorphism 
\begin{equation} \label{equiviso}
\mathcal{X} \setminus \pi^{-1}(0) \simeq X \times (\mathbb{P}^1 \setminus \{0\})
\end{equation}
\item 
a $\mathbb{C}^*$-invariant $(1,1)-$Bott-Chern cohomology class $\mathcal{A} \in  H^{1,1}_{\mathrm{BC}}(\mathcal{X},\mathbb{R})^{\mathbb{C}^*}$ whose image under \eqref{equiviso}
is $p_1^*\alpha$. 
\end{itemize}
\end{mydef}

\begin{rem}
Note that since $\pi$ is flat the central fiber $\mathcal{X}_0 := \pi^{-1}(0)$ is a Cartier divisor, so $\mathcal{X} \setminus \mathcal{X}_0$ is dense in $\mathcal{X}$ in Zariski topology. 
\end{rem}

\noindent The \emph{trivial} test configuration for $(X,\alpha)$ is given by $(\mathcal{X} := X \times \mathbb{P}^1, p_1^*\alpha, \lambda_{\mathrm{triv}}, p_2)$, where $p_1: X \times \mathbb{P}^1 \rightarrow \mathbb{P}^1$ and $p_2: X \times \mathbb{P}^1 \rightarrow \mathbb{P}^1$ are the projections on the $1^{\mathrm{st}}$ and $2^{\mathrm{nd}}$ factor respectively, and $\lambda_{\mathrm{triv}}: \mathbb{C}^* \times \mathcal{X} \rightarrow \mathcal{X}$, $(\tau, (x,z)) \mapsto (x, \tau z)$ is the $\mathbb{C}^*$-action that acts trivially on the first factor. If we instead let $\sigma: \mathbb{C}^* \times X \rightarrow X$ be any $\mathbb{C}^*$-action on $X$, then we obtain an induced test configuration as above with $\lambda(\tau, (x,z)) := (\sigma(\tau,x), \tau z)$ (by also taking the compactification so that the fiber at inifinity is trivial). Such test configurations are called \emph{product} test configurations of $(X,\alpha)$.

In either case, we identify $X$ with $X \times \{ 1 \}$ and the canonical equivariant isomorphism \eqref{equiviso} is then explicitly induced by the isomorphisms $X \cong X \times \{1\} \rightarrow X \times \{ \tau \}$ given by $x \mapsto \lambda(\tau, (x,1)) =: \lambda(\tau) \cdot x$. 

\begin{mydef} \label{Definition smooth and dom} A cohomological test configuration $(\mathcal{X},\mathcal{A})$ is said to be smooth (resp. dominating) if $\mathcal{X}$ is smooth (resp. $\mathcal{X}$ dominates $X \times \mathbb{P}^1$ via a morphism $\mu: \mathcal{X} \rightarrow X \times \mathbb{P}^1$. It is said to be \emph{relatively K\"ahler} if $\mathcal{A}$ is a relatively K\"ahler class, i.e. if there is a K\"ahler form $\beta \in H^{1,1}(\mathbb{P}^1)$ such that $\mathcal{A} + \pi^*\beta$ is K\"ahler on $\mathcal{X}$. 
A \emph{pull-back} of $(\mathcal{X},\mathcal{A}, \lambda, \pi)$ via a morphism $\mu: \hat{\mathcal{X}} \rightarrow \mathcal{X}$ is any test configuration $(\hat{\mathcal{X}}, \mu^*\mathcal{A}, \hat{\lambda}, \pi \circ \mu)$, where the $\mathbb{C}^*$-action $\hat{\lambda}$ may be taken arbitrarily (but so that the data still defines a test configuration). Hence, we do not demand any uniqueness here. 
\end{mydef}

\noindent We give a few remarks and examples on how to compare cohomological test configurations with \emph{algebraic test configurations $(\mathcal{X},\mathcal{L})$ for a polarized manifold $(X,L)$}. The latter refers to the well-known concept first introduced in \cite{Donaldsontoric}. 
\begin{enumerate}
\item If $(X,L)$ is any compact K\"ahler manifold endowed with an ample line bundle $L$ (so $X$ is projective), and $(\mathcal{X}, \mathcal{L})$ is a test configuration for $(X,L)$ in the usual algebraic sense, cf. e.g. \cite{Donaldsontoric}, then $(\mathcal{X}, c_1(\mathcal{L}))$ is a cohomological test configuration for $(X, c_1(L))$. This observation is useful, since many examples of algebraic test configurations $(\mathcal{X}, \mathcal{L})$ for polarized manifolds $(X,L)$ are known, see e.g. \cite{Gabornotes, Tianlecturenotes} and references therein.

\item There are more cohomological test configurations for $(X, c_1(L))$ than there are algebraic test configurations for $(X,L)$ (take for instance $(\mathcal{X}, \mathcal{A})$ with $\mathcal{A}$ a transcendental class as in the above definition), but in some cases the ensuing stability notions can nonetheless be seen to be equivalent (see \cite[Section 3]{SD1}).
\end{enumerate}

\subsubsection{The Donaldson-Futaki invariant and the non-Archimedean Mabuchi functional}

\begin{mydef} \emph{To any cohomological test configuration $(\mathcal{X},\mathcal{A})$ for $(X,\alpha)$ we may associate its \emph{Donaldson-Futaki invariant} $\mathrm{DF}(\mathcal{X},\mathcal{A})$ and its \emph{non-Archimedean Mabuchi functional} ${{\mathrm{M}}^{\mathrm{NA}}}(\mathcal{X},\mathcal{A})$, first introduced in \cite{BHJ1}. They are given respectively by the following intersection numbers
\begin{equation*}
\mathrm{DF}(\mathcal{X},\mathcal{A}) := \frac{\bar{\mathcal{S}}}{n+1}V^{-1} (\mathcal{A}^{n+1})_{\hat{\mathcal{X}}} + V^{-1}(K_{\mathcal{X}/\mathbb{P}^1} \cdot \mathcal{A}^n)_{\hat{\mathcal{X}}}
\end{equation*}
and
$$
{{\mathrm{M}}^{\mathrm{NA}}}(\mathcal{X},\mathcal{A}) := \mathrm{DF}(\mathcal{X},\mathcal{A}) + ((\mathcal{X}_{0,\mathrm{red}} - \mathcal{X}_0) \cdot \mathcal{A}^n)_{\hat{\mathcal{X}}}
$$
computed on any smooth and dominating model $\tilde{\mathcal{X}}$ of $\mathcal{X}$ (due to the projection formula it does not matter which one). 
Note that $\mathrm{DF}(\mathcal{X},\mathcal{A}) \geq {{\mathrm{M}}^{\mathrm{NA}}}(\mathcal{X},\mathcal{A})$ with equality precisely when $\mathcal{X}_0$ is reduced. }
\end{mydef}

\noindent In case $\mathcal{X}$ is smooth, $K_{\mathcal{X}/\mathbb{P}^1} := K_{\mathcal{X}} - \pi^*K_{\mathbb{P}^1}$ denotes the relative canonical class taken with respect to the flat morphism $\pi:\mathcal{X} \rightarrow \mathbb{P}^1$. In the general case of a normal (possibly singular) test configuration $\mathcal{X}$ for $X$, we need to give meaning to the intersection number $K_{\mathcal{X}} \cdot \mathcal{A}_1 \cdot \dots \cdot \mathcal{A}_n$, for $\mathcal{A}_i \in H^{1,1}_{\mathrm{BC}}(\mathcal{X},\mathbb{R})$. To do this, suppose that $\tilde{\mathcal{X}}$ is a smooth model for $\mathcal{X}$, with $\pi': \tilde{\mathcal{X}} \rightarrow \mathcal{X}$ the associated morphism. Since $\tilde{\mathcal{X}}$ is smooth the canonical class $K_{\tilde{\mathcal{X}}} := \omega_{\tilde{\mathcal{X}}}$ is a line bundle. Now 
consider 
$
\omega_{\mathcal{X}} := \mathcal{O}(K_{\mathcal{X}}) := (\pi'_*\omega_{\tilde{\mathcal{X}}})^{**},
$
i.e. the "reflexive extension" of $\omega_{\tilde{\mathcal{X}}}$, which is a rank $1$ reflexive sheaf on $\mathcal{X}$. Having been unable to find a suitable reference for an intersection theory for reflexive sheaves, we here content ourselves with an "ad hoc" definition of the concerned intersection number, setting 
\begin{equation}
(\omega_{\mathcal{X}} \cdot \mathcal{A}_1 \cdot \dots \cdot \mathcal{A}_n) := ( K_{\tilde{\mathcal{X}}} \cdot \pi'^*\mathcal{A}_1 \cdot \dots \cdot \pi'^*\mathcal{A}_n).
\end{equation}
Using the projection formula, it is straightforward to see that the above intersection number (hence also $\mathrm{DF}$ and ${{\mathrm{M}}^{\mathrm{NA}}}$) is independent of the choice of model/resolution $\pi': \tilde{\mathcal{X}} \rightarrow \mathcal{X}$. 


\subsubsection{Product test configurations and definition of K-polystability}

\begin{mydef} \label{Def Kps} \emph{In analogy with the usual definition for polarized manifolds, and following \cite[Section 3]{SD1}, we say that }

\begin{itemize}
\item \emph{$(X,\alpha)$ is \emph{K-semistable} if $\mathrm{DF}(\mathcal{X},\mathcal{A}) \geq 0$ for all normal and relatively K\"ahler test configurations $(\mathcal{X}, \mathcal{A})$ for $(X,\alpha)$. }
\item \emph{$(X,\alpha)$ is \emph{K-polystable} if it is K-semistable, and in addition $\mathrm{DF}(\mathcal{X},\mathcal{A}) = 0$ if and only if $\mathcal{X}_{\vert \pi^{-1}(\mathbb{C})}$ is $\mathbb{C}^*$-equivariantly isomorphic to $X \times \mathbb{C}$. }
\end{itemize}
\end{mydef}

\noindent Note that demanding that $\mathcal{X}$ is $\mathbb{C}^*$-equivariantly isomorphic to $X \times \mathbb{P}^1$ is not enough: For instance, there are (algebraic) product test configurations $(X,L) \times \mathbb{C}$ whose Donaldson-Futaki invariant vanishes, but whose compactifications over $\mathbb{P}^1$ (and thus their corresponding cohomological test configuration $(\bar{\mathcal{X}}, c_1(\bar{\mathcal{L}}))$) is \emph{not} a product. See e.g. \cite[Example 2.8]{BHJ1}.

When it is necessary to make the distinction, we will refer to the above stability notions as \emph{cohomological}. In the same vein, we refer to the analogous stability notions for polarized manifolds (see e.g. \cite{Donaldsontoric, Berman, BHJ1}) as \emph{algebraic}.
For future use we recall the following typical construction of product configurations, given a $\mathbb{C}^*$-action on our given manifold $X$. 

\begin{ex} \label{Example product}
Any product test configuration is given by $X \times \mathbb{P}^1$ with the $\mathbb{C}^*$-action given as the diagonal action induced by a one-parameter subgroup $\lambda: \mathbb{C}^* \rightarrow \mathrm{Aut}(X)$, i.e. $\tau \cdot (x,z) := (\lambda(\tau)\cdot x, \tau z)$. The canonical morphism $\mu: X \times \mathbb{P}^1 \rightarrow X \times \mathbb{P}^1$ associated to the induced product configuration is given by $(x,z) \mapsto (\lambda^{-1}(z) \cdot x, z)$. 
\end{ex}

\noindent It was proven in \cite[Theorem A]{SD1} that cscK manifolds are always K-semistable. Moreover, note that if $(X,L)$ is a polarized manifold, then $(X,L)$ is K-semistable in the usual algebraic sense iff $(X, c_1(L))$ are (cohomologically) K-semistable. In other words, the algebraic and the cohomological notions of K-semistability are equivalent. 
It is an open question whether the same holds for K-polystability, but at least one of the implications always holds:

\begin{prop}
Suppose that $(X,L)$ is a polarized manifold such that $(X,c_1(L))$ is cohomologically K-polystable. Then $(X,L)$ is algebraically K-polystable.
\end{prop}

\begin{proof}
Suppose that $(X,c_1(L))$ is cohomologically K-polystable and let $(\mathcal{X},\mathcal{L})$ be an algebraic test configuration for $(X,L)$. Then $(\mathcal{X},c_1(\mathcal{L}))$ is a cohomological test configuration for $(X,c_1(L))$ which moreover satisfies $\mathrm{DF}(\mathcal{X},\mathcal{L}) = \mathrm{DF}(\mathcal{X},c_1(\mathcal{L}))$. By assumption  we conclude that for all relatively K\"ahler test configurations $\mathrm{DF}(\mathcal{X},\mathcal{L}) \geq 0$, with equality if and only if $(\mathcal{X},c_1(\mathcal{L}))$ is a product. Finally, $(\mathcal{X},c_1(\mathcal{L}))$ is a product if and only if $(\mathcal{X},\mathcal{L})$ is a product, so we are done. 
\end{proof} 

\noindent In particular, the above notion of K-polystability generalizes the usual notion for polarized manifolds considered in \cite{Berman, BDL}.

\bigskip

\section{Asymptotics of Deligne functionals along singular test configurations} \label{Section 3}

\noindent In this section we study subgeodesic rays with singularity type prescribed by a (possibly singular) test configuration, and prove that the asymptotic slope of Deligne functionals against these rays can be computed as intersection numbers on the total space of the test configuration. We furthermore treat the case of the K-energy functional (Theorem \ref{Theorem generalized thm C}). We finally show (Lemma \ref{Lemma unique model}) that the singularity type of the associated ray actually characterizes the test configuration up to $\mathbb{C}^*$-equivariant isomorphism of the total space.

\subsection{Subgeodesic and geodesic rays associated to a cohomological test configuration} \label{Section geodesic}

\noindent Fix a compact K\"ahler manifold $(X,\omega)$ and let  $(\varphi_t)_{t \geq 0} \subset \mathrm{PSH}(X,\omega)$ be a ray of $\omega$-psh functions on $X$.  
Following Donaldson \cite{Donaldsontoric} and Semmes \cite{Semmes} we consider the standard correspondence between the family $(\varphi_t)_{t \geq 0}$ and an associated $S^1$-invariant function $\Phi$ on  $X \times \bar{\Delta}^*$, given by
$
\Phi(x,e^{-t+is}) = \varphi^t(x),
$
where the sign is chosen so that $t \rightarrow +\infty$ corresponds to $\tau := e^{-t+is} \rightarrow 0$. Here $\bar{\Delta}^*\subset \mathbb{C}$ denotes the punctured unit disc. The function $\Phi$ restricted to a fiber $X \times \{\tau\}$ thus corresponds precisely to $\varphi_t$ on $X$. In the direction of the fibers we hence have $p_1^*\omega + dd^c_{x}\Phi \geq 0$ (in the sense of currents, letting  $p_1: X \times \bar{\Delta}^* \rightarrow X$ denote the first projection).

\begin{mydef} \label{subgeodesic definition}
We say that $(\varphi_t)_{t \geq 0}$ is a \emph{subgeodesic ray} if the associated $S^1$-invariant function $\Phi$ on $X \times \bar{\Delta}^*$ is $p_1^*\omega$-psh. Furthermore, a locally bounded family of functions $(\varphi_t)_{t \geq 0}$ in $\mathrm{PSH}(X,\omega)$ is said to be a \emph{weak geodesic ray} if the associated $S^1$-invariant function $\Phi \in \mathrm{PSH}(X \times \bar{\Delta}^*, p_1^*\omega)$ 
satisfies
\begin{equation*}
(p_1^*\omega + dd^c_{(x,\tau)}\Phi)^{n+1} = 0
\end{equation*}
on $X \times \Delta^*$. 
\end{mydef}
\begin{mydef}
We say that $(\varphi^t)_{t \geq 0}$ is continuous (resp. locally bounded, smooth) if the corresponding $S^1$-invariant function $\Phi$ is continuous (resp. locally bounded, smooth). 
\end{mydef}

\noindent This terminology is motivated by the extensive study of (weak) geodesics in the space $\mathcal{H}$, 
see e.g. \cite{Blocki13, Chen00, Darvas14, Donaldsontoric, Semmes, Tosattigeodesicregularity}. 

Relying on theory for degenerate complex Monge-Amp\`ere equations on manifolds with boundary it is possible to (uniquely) associate a weak geodesic ray to any given \emph{smooth} cohomological test configuration $(\mathcal{X},\mathcal{A})$ for $(X,\alpha)$ \emph{dominating $X \times \mathbb{P}^1$}. We refer to \cite[Lemma 4.6]{SD1} for the construction of this geodesic ray. It is always of regularity $\mathcal{C}^{1,1}$, see \cite{Tosattigeodesicregularity}.

\medskip

\subsection{Subgeodesic rays with prescribed singularity type} \label{Subsection associated rays}

Now let $(\varphi_t)$ and $(\psi_t)$ be subgeodesic rays. Then there are unique $p_1^*\omega$-psh extensions to $X \times \bar{\Delta}$ of the $S^1$-invariant $p_1^*\omega$-psh functions corresponding to $(\varphi_t)$ and $(\psi_t)$, still denoted by $\Phi$ and $\Psi$. We then say that 
$$
(\varphi_t) \sim (\psi_t) \Longleftrightarrow \Phi - \Psi \in L^{\infty}_{\mathrm{loc}}(X \times \Delta)
$$
i.e. iff the rays have the same singularity type. We write $[(\varphi_t)]$ for the equivalence class of $(\varphi_t)$ under this relation. Compare also the notion of "parallell rays" in \cite{ChenTang}. 

In the following paragraphs we introduce a compatibility condition between subgeodesic rays and test configurations. As explained below this yields a natural way to associate a singularity type class of subgeodesic rays to any given test configuration $(\mathcal{X},\mathcal{A})$ for $(X,\alpha)$, in a way that is adapted to applications for energy functional asymptotics, see Section \ref{Section energy functional asymptotics}.

\subsubsection{Compatible rays and the singularity type of a test configuration} 
We here extend the definition given in \cite[Section 4]{SD1} to singular test configurations.  
To this end, suppose that $(\mathcal{X},\mathcal{A})$ is a (possibly singular) relatively K\"ahler test configuration for $(X,\alpha)$. Consider a smooth model $\hat{\mathcal{X}}$ for $\mathcal{X}$, i.e. a a $\mathbb{C}^*$-equivariant bimeromorphic morphism $\rho: \hat{\mathcal{X}} \rightarrow \mathcal{X}$, where $\hat{\mathcal{X}}$ is smooth and moreover dominates the product $X \times \mathbb{P}^1$. 
Note that a smooth model always exists (consider e.g. normalization of the graph of $\mathcal{X} \dashrightarrow X \times \mathbb{P}^1$ and resolve singularities).
We then have the following situation:

\[
 \begin{tikzcd}[ampersand replacement=\&]
    \hat{\mathcal{X}} \arrow[d, "\rho"] \arrow[ddr, controls={+(-1.5,-2) and +(0,0)},  "\pi", swap]
    \ar[dr, "\mu"]\\
    \mathcal{X} \arrow[dashed, r] \& X \times \mathbb{P}^1 \ar[r,"p_1"] \ar[d, "p_2"] \& X \\
    \& \mathbb{P}^1
 \end{tikzcd}
\]

\noindent Now suppose that $(\varphi_t)$ is any locally bounded subgeodesic ray on $X$, with $\Phi$ the $S^1$-invariant function on $X \times \mathbb{P}^1$ associated to the given ray $(\varphi_t)_{t \geq 0}$. By \cite[Proposition 3.10]{SD1} we have  
$$
\rho^*\mathcal{A} = \mu^*p_1^*\alpha + [D],
$$
where $D = \sum_{j=1}^n a_i D_i$ is a divisor on $\hat{\mathcal{X}}$ supported on the central fiber $\hat{\mathcal{X}}_0$. We can decompose the current of integration of $D$ as $\delta_D = \theta_D + dd^c\psi_D$, where $\theta_D$ is a smooth $S^1$-invariant $(1,1)$-form on $\hat{\mathcal{X}}$. Locally, we thus have 
$$
\psi_D = \sum_j a_j \log|f_j| \; \; \textrm{mod} \; \mathcal{C}^{\infty}, 
$$ 
where (writing $D := \sum_j a_jD_j$ for the decomposition of $D$ into irreducible components) the $f_j$ are local defining equations for the $D_j$ respectively. 
In particular, the choice of $\psi_D$ is uniquely determined modulo a smooth function on $\hat{\mathcal{X}}$.

\begin{mydef} \label{Definition compatibility} A locally bounded subgeodesic ray $(\varphi_t)_{t \geq 0}$ on $X$ is said to be $L^{\infty}$-compatible with $(\mathcal{X}, \mathcal{A})$ if 
$
\Psi := \Phi \circ \mu + \psi_D
$
extends to a locally bounded $\rho^*\Omega$-psh function on $\hat{\mathcal{X}}$. 
\end{mydef}

\noindent \noindent In particular, the singularity type of $\Phi \circ \mu$ is determined by the Green function $\psi_D$. The following notion is also useful: 

\begin{mydef} \label{Definition compatibility smooth}
A smooth ray $(\varphi_t)_{t \geq 0}$ is $C^{\infty}$-compatible with $(\mathcal{X}, \mathcal{A})$ if 
$
\Psi := \Phi \circ \mu + \psi_D
$
extends smoothly across $\hat{\mathcal{X}}_0$.
\end{mydef}

\begin{ex} \emph{A few examples are in order. Let $(\mathcal{X},\mathcal{A})$ be a given test configuration for $(X,\alpha)$.}
\begin{itemize}
\item \emph{The geodesic ray $(\varphi_t)^{(\mathcal{X},\mathcal{A})}$ is $L^{\infty}$-compatible but in general not $C^{\infty}$-compatible with $(\mathcal{X},\mathcal{A})$. See \cite[Section 4]{SD1} for details on the construction. }
\item \emph{Let $\Omega$ be a smooth $S^1$-invariant $(1,1)$-form such that $[\Omega] = \mathcal{A}$. Denote by $\Omega_{\tau}$ the restriction of $\Omega$ to the fiber $\mathcal{X}_{\tau}$. As noted in \cite{DervanRoss}, $\Omega_{\tau}$ and $\Omega_1$ are cohomologous, so we may define a ray $(\varphi_t)$ by the relation }
$
\lambda(\tau)^*\Omega_{\tau} - \Omega_1 = dd^c\varphi_{\tau}. 
$
\emph{We claim that the ray $(\psi_t)$ on $X$ defined by the relation $\psi_t := \varphi_{e^{-t}}$ is smooth and $C^{\infty}$-compatible with $(\mathcal{X},\mathcal{A})$. However, it is not in general a subgeodesic. }
\end{itemize}
\end{ex}

\noindent In particular, we observe that the set of subgeodesic rays $L^{\infty}$-compatible with a given test configuration are precisely those that are of the same singularity type as the uniquely associated geodesic ray. Interestingly, this is not necessarily true for $C^{\infty}$-compatibility. While the geodesic ray associated to a test configuration is a natural object, we will see in Section \ref{Section energy functional asymptotics} below that it is often useful to consider the more flexible compatibility notions rather than the geodesic ray (or subgeodesic rays of same singularity type).

\medskip

\subsection{Asymptotics of the Mabuchi- and Deligne functionals along singular test configurations } \label{Section energy functional asymptotics}
Let $(\mathcal{X},\mathcal{A})$ be a possibly singular test coniguration for $(X,\alpha)$. By resolution of singularities $\rho: \hat{\mathcal{X}} \rightarrow \mathcal{X}$ and the projection formula, it suffices to consider $(\hat{\mathcal{X}}, \rho^*\mathcal{A})$. Note that $\rho^*\mathcal{A}$ is here \emph{relatively nef} (with the loss of positivity occurring over the central fiber $\mathcal{X}_0$. Along these lines, reating the case of singular test configurations reduces to the case of smooth and \emph{relatively nef} test configurations. Using this point-of-view we now set out to extend \cite[Corollary 4.14 and Theorem 5.1]{SD1} by computing the asymptotic slope of the $\mathrm{J}$-functional and the Mabuchi functional respectively, considered along rays associated to possibly singular (normal) cohomological test configurations.

\subsubsection{Asymptotics for Deligne functionals}
For future use we first note that \cite[Theorem B]{SD1} still holds for normal (possibly singular) test configurations, since indeed neither the lhs or the rhs sees the central fiber (due to the projection formula). In other words, one may pass to any smooth model and the result goes through as in the smooth case. 
It should be understood that the following intersection numbers are computed on resolutions (see e.g. Remark \ref{Remark intersection number} below). 

\begin{thm} \label{Theorem generalized theorem B}
Let $X$ be a compact K\"ahler manifold of dimension $n$ and let $\theta_i$, $0 \leq i \leq n$, be closed $(1,1)$-forms on $X$. Set $\alpha_i := [\theta_i] \in H^{1,1}(X,\mathbb{R})$. Consider relatively K\"ahler cohomological test configurations $(\mathcal{X}, \mathcal{A}_i)$ for $(X,\alpha_i)$.  For each collection of smooth rays $(\varphi_i^t)_{t \geq 0}$ $\mathcal{C}^{\infty}$-compatible with $(\mathcal{X}_i, \mathcal{A}_i)$ respectively, the asymptotic slope of the multivariate energy functional $\langle \cdot, \dots, \cdot \rangle := \langle \cdot, \dots, \cdot \rangle_{(\theta_0, \dots, \theta_n)}$ is well-defined and satisfies
\begin{equation*}
\frac{\langle \varphi_0^t, \dots, \varphi_n^t \rangle}{t} \longrightarrow (\mathcal{A}_0 \cdot \dots \cdot \mathcal{A}_n)
\end{equation*}
as $t \rightarrow +\infty$. 
\end{thm}

\begin{rem} \label{Remark intersection number}
Here the intersection number in the right hand side is computed on a common smooth and dominating model $\rho_i: \hat{\mathcal{X}} \rightarrow \mathcal{X}_i$. We then set
$$
(\mathcal{A}_0 \cdot \dots \cdot \mathcal{A}_n) := (\rho_0^*\mathcal{A}_0, \cdot \dots \cdot \rho_n^*\mathcal{A}_n).
$$   
This intersection number is independent of the choice of $\hat{\mathcal{X}}$. 
\end{rem}

\noindent The point is to note that this generalized version of \cite[Theorem B]{SD1} holds also for singular test configurations. 

\begin{proof}
Since the neither the left hand side or the right hand side of the statement sees the central fiber, the proof goes through as in the smooth and dominating case (see \cite[Theorem B]{SD1}). For the reader's convenience we recall the argument, with the required adaptations in notation:  Fix any smooth $S^1$-invariant $(1,1)$-forms $\Omega_i$  on $\mathcal{X}_i$ such that $[\Omega_i] =\mathcal{A}_i$ in $H^{1,1}(\mathcal{X}_i, \mathbb{R})$. Let $(\varphi_i^t)_{t\geq 0}$ be smooth and $\mathcal{C}^{\infty}$-compatible with $(\mathcal{X}, \mathcal{A}_i)$ respectively. Now fix a smooth model $\rho: \hat{\mathcal{X}} \rightarrow \mathcal{X}$ that also dominates $X \times \mathbb{P}^1$ via a morphism $\mu: \hat{\mathcal{X}} \rightarrow X \times \mathbb{P}^1$. 
In the notation of Definition \ref{Definition compatibility} the functions $\Phi_i\circ\mu+\psi_D$ are then smooth on the manifold with boundary $M:=\pi^{-1}(\bar{\Delta})$, and may thus be written as the restriction of smooth $S^1$-invariant functions $\Psi_i$ on $\hat{\mathcal{X}}$ respectively. 
Using the $\mathbb{C}^*$-equivariant isomorphism $\hat{\mathcal{X}}\setminus\hat{\mathcal{X}}_0\simeq X\times(\mathbb{P}^1\setminus\{0\})$ we view $(\Psi_i-\psi_D)_{\vert \hat{\mathcal{X}}_{\tau}}$ as a function $\varphi_i^{\tau}\in \mathcal{C}^{\infty}(X)$. 
By \cite[Proposition 2.8]{SD1} we then have
$$
dd^c_{\tau} \langle \varphi_0^t, \dots, \varphi_n^t \rangle =\pi_*\left(\bigwedge_i (\Omega_i + dd^c\Psi_i)\right).
$$
over $\mathbb{P}^1\setminus\{0\}$. 
Denoting by $u(\tau) := \langle \varphi_0^{\tau}, \dots, \varphi_n^{\tau} \rangle$ the Green-Riesz formula  then yields
$$
\frac{d}{dt}_{t=-\log\varepsilon} u(\tau)=\int_{\mathbb{P}^1\setminus\Delta_{\varepsilon}} dd^c_{\tau} u(\tau)=
$$
$$
\int_{\pi^{-1}(\mathbb{P}^1\setminus\Delta_{\varepsilon})}\bigwedge_i (\Omega_i + dd^c\Psi_i), 
$$
which converges to $(\mathcal{A}_0 \cdot \dots \cdot \mathcal{A}_n)$ as $\varepsilon\rightarrow 0$.  

It remains to show that 
\begin{equation*} 
 \lim_{t \rightarrow +\infty} \frac{ u(\tau)}{t} = \lim_{t \rightarrow +\infty} \frac{d}{dt} u(\tau),
\end{equation*} 
To see this, note that for each  closed $(1,1)$-form $\Theta$ on $\mathcal{X}$ and each smooth function $\Phi$ on $\mathcal{X}$, there is a K\"ahler form $\eta$ on $\mathcal{X}$ and a constant $C$ large enough so that $\Theta + C\eta + dd^c\Phi \geq 0$ on $\mathcal{X}$. Moreover, we have a relation
$$
\langle \varphi_0^t, \varphi_1^t, \dots, \varphi_n^t \rangle_{(\omega - \omega', \theta_1 \dots, \theta_n)}  = 
$$

$$
\langle {\varphi_0^t}, \varphi_1^t \dots, \varphi_n^t \rangle_{(\omega, \theta_1, \dots, \theta_n)} - \langle 0, \varphi_1^t \dots, \varphi_n^t \rangle_{(\omega', \theta_1, \dots, \theta_n)}
$$ 


\noindent and repeat this argument for each $i$, $0 \leq i \leq n$, by symmetry. It follows from the above 'multilinearity' that we can write $t \mapsto \langle \varphi_0^t, \dots, \varphi_n^t \rangle$ as a difference of convex functions,
concluding the proof. 
\end{proof}

\noindent The following is a generalization of \cite[Corollary 4.14]{SD1}: 

\begin{cor}  Let $(\mathcal{X},\mathcal{A})$ be a normal test configuration for $(X,\alpha)$ and let $t \mapsto \varphi_t$ be a \emph{subgeodesic} ray $L^{\infty}$-compatible with $(\mathcal{X},\mathcal{A})$. Suppose that $\hat{\mathcal{X}}$ is any smooth and dominating test configuration, with $\rho: \hat{\mathcal{X}} \rightarrow \mathcal{X}$ the associated morphism. Then the following limit is well-defined and 
$$
\frac{\mathrm{J}(\varphi_t)}{t} \longrightarrow \mathrm{J}^{\mathrm{NA}}(\mathcal{X},\mathcal{A}),
$$
as $t \rightarrow +\infty$. 
Here  $$\mathrm{J}^{\mathrm{NA}}(\mathcal{X},\mathcal{A}) := V^{-1}\left\{(\rho^*\mathcal{A} \cdot \mu^*p_1^*\alpha^n) - \frac{(\rho^*\mathcal{A}^{n+1})}{n+1}\right\}$$  
\end{cor}

\begin{rem}
As before, the intersection number $\mathrm{J}^{\mathrm{NA}}(\mathcal{X},\mathcal{A})$ is independent of the choice of smooth model $\hat{\mathcal{X}}$.
\end{rem}


\subsubsection{Asymptotics for the Mabuchi functional}
Along the same lines we see that the asymptotics of the K-energy functional can be estimated along $C^{\infty}$-compatible rays (in case the test configuration is smooth we can also consider so called $C^{1,1}$-compatible rays, cf. \cite[Section 4]{SD1}, but we will not need this here). 
What we will be using is the following version of \cite[Theorem 5.1]{SD1}, valid also for singular (albeit normal) cohomological test configurations. As before, it should be understood that all intersection numbers are computed on any resolution, as in the definition of $C^{\infty}$- and $L^{\infty}$-compatibility, see Definition \ref{Definition compatibility}.

\begin{thm} \label{Theorem generalized thm C}
Let $(\mathcal{X},\mathcal{A})$ be a normal and relatively K\"ahler test configuration for $(X,\alpha)$. Let $\varphi_0 \in \mathcal{H}_0$. Then there is a smooth ray $[0,+\infty[ \ni t \mapsto \psi_t \in \mathrm{PSH}(X,\omega) \cap C^{\infty}(X) \cap E^{-1}(0)$ on $X$ emanating from $\varphi_0$, that is $C^{\infty}$-compatible with $(\mathcal{X},\mathcal{A})$ and satisfies
$$
\lim_{t \rightarrow +\infty} \frac{{\mathrm{M}}(\psi_t)}{t} = {{\mathrm{M}}^{\mathrm{NA}}}(\mathcal{X},\mathcal{A}).
$$
\end{thm}



\begin{proof}
We begin by constructing such a ray $r \mapsto \psi_t$. To do this, consider any smooth model $\hat{\mathcal{X}}$ for $\mathcal{X}$, with $\rho: \hat{\mathcal{X}} \rightarrow \mathcal{X}$ the associated morphism. Since $\mathcal{A}$ is relatively K\"ahler on $\mathcal{X}$, note that $\hat{\mathcal{A}} := \rho^*\mathcal{A}$ is relatively nef on $\hat{\mathcal{X}}$ (with the loss of positivity occurring over the central fiber $\hat{\mathcal{X}}_0$). Now let $\Omega$ be any smooth $S^1$-invariant $(1,1)$ form on $\mathcal{X}$ such that $[\Omega] = \mathcal{A}$. For each $\tau \in \mathbb{C}^*$ we write $\rho(\tau): \hat{\mathcal{X}}_{\tau} \rightarrow \mathcal{X}_{\tau}$ for the isomorphism between the respective fibers, and $\lambda(\tau): \mathcal{X}_{\tau} \rightarrow X$. In particular, we identify $X$ with $\mathcal{X}_1$ and $\hat{\mathcal{X}}_1$, via the respective isomorphisms $\lambda(1)$ and $\lambda(1) \circ \rho(1)$, and write $\hat{\lambda}(\tau): \hat{\mathcal{X}}_{\tau} \rightarrow X$ for the composition $\lambda(\tau) \circ \rho(\tau)$. 

Now note that $\rho(\tau)^*\Omega_{\tau} := \hat{\Omega}_{\tau}$ and $\rho(1)^*\Omega_1 := \hat{\mathcal{X}}_{1}$ are cohomologous, so there is for each $\tau \neq 0$ a smooth real $S^1$-invariant function $\xi_{\tau}$ on $X$, for which 
$$
\hat{\lambda}(\tau)^*\hat{\Omega}_{\tau} = \hat{\lambda}(1)^*\hat{\Omega}_1 + dd^c\xi_{\tau}.  
$$
Then $t \mapsto \psi_t := \xi_{\tau}$, with $\tau := e^{-t + is}$, defines a smooth ray (without loss of generality normalized so that $E(\varphi_t) = 0$) that is moreover $C^{\infty}$-compatible with $(\hat{\mathcal{X}}, \hat{\mathcal{A}})$, thus also with $(\mathcal{X},\mathcal{A})$ by Definition \ref{Definition compatibility smooth}. We finally claim that the proof of \cite[Theorem 5.6]{SD1} can be seen to go through for this smooth ray $(\psi_t)$. Indeed $\hat{\Omega}_1 + dd^c\xi_{\tau}$ is strictly positive away from $\hat{\mathcal{X}}_0$, and neither the lhs nor the rhs (due to the projection formula) sees the central fiber. 
\end{proof}

\medskip

\subsection{An injectivity lemma for the singularity type class}

We finally prove an injectivity result 
that is central for our proof of geodesic K-polystability of cscK manifolds. Consider the assignment
\begin{equation}
(\mathcal{X},\mathcal{A}) \mapsto [(\varphi_t)^{(\mathcal{X},\mathcal{A})}]
\end{equation}
which sends a relatively K\"ahler test configuration to the equivalence class of subgeodesic rays with which it is $L^{\infty}$-compatible (Definition \ref{Definition compatibility}). We wish to prove that if $(\mathcal{X},\mathcal{A})$ and $(\mathcal{Y},\mathcal{B})$ are two test configuration for $(X,\alpha)$ satisfying 
$
[(\varphi_t)^{(\mathcal{X},\mathcal{A})}] = [(\varphi_t)^{(\mathcal{Y},\mathcal{B})}],
$
then the canonical $\mathbb{C}^*$-equivariant isomorphism $\mathcal{X} \setminus \mathcal{X}_0 \rightarrow \mathcal{Y} \setminus \mathcal{Y}_0$ extends to an isomorphism $\mathcal{X} \rightarrow \mathcal{Y}$. 
This can be seen as an extension of the concept of unique ample model, cf. \cite{BHJ1}.

\subsubsection{Uniqueness of the relatively K\"ahler model} \label{Section unique model} 
We now introduce the
 concept of \emph{relatively K\"ahler model} and show that such an object is \emph{unique}. The methods used are  completely different from the ones in the case of polarized manifolds, when the corresponding result follows from the existence of a one-to-one correspondence between ample test configurations and finitely generated $\mathbb{Z}$-filtrations. This follows from the so called 'reverse Rees construction', see e.g. \cite{BHJ1} for details. 

To introduce our methods, first suppose that $X$, $Y$, $Z$ are normal compact K\"ahler spaces
such that $\phi: X \dashrightarrow Y$ is a bimeromorphic map and $\mu: Z \rightarrow X$, $\rho: Z \rightarrow Y$ are bimeromorphic morphisms (modifications). Up to replacing $Z$ by a $Z'$ we can suppose that $\mu$ is a sequence of blow-ups with smooth center (see Hironaka \cite{Hironaka1, AHV1, AHV2}). In particular, $\mu$ is a projective morphism. Importantly, this implies that the fibers $\mu^{-1}(x)$, $x \in X$ are projective varieties, so they are covered by curves. 
\[
 \begin{tikzcd}[ampersand replacement=\&]
    Z \arrow[d, "\mu"] \arrow[dr, "\rho"]\\
    X \arrow[dashed, r] \& Y 
 \end{tikzcd}
\]  
Assume further that $\alpha \in H^{1,1}(X,\mathbb{R})$ and $\beta \in H^{1,1}(Y,\mathbb{R})$ are K\"ahler classes satisfying
$
\mu^*\alpha = \rho^*\beta.
$
We then claim that $\phi$ is in fact an isomorphism. Indeed, let $x \in X$, and $C$ be a curve in $\mu^{-1}(x) \subset Z$. The projection formula then yields
$$
0 = (\mu^*\alpha) \cdot [C] = (\rho^*\beta) \cdot [C] = \beta \cdot \rho_*[C]. 
$$
Since $\beta$ is K\"ahler we must have $\rho_*[C] = 0$, so that $\dim \rho(C) < 1$. Hence $\rho$ contracts the curve $C$ to a point in $Y$. Finally recall that $\phi$ is a morphism if and only if for all curves $C \subset Z$,  $\mu(C) = \mathrm{point}$ implies that $\rho(C) = \mathrm{point}$. By symmetry in $X$ and $Y$ (in particular, $\phi$ is bimeromorphic and $\alpha$ is also K\"ahler) this concludes. 

The point is that, along the lines of the above, we obtain the following key lemma.
It can be seen as a K\"ahler analogue of the notion of \emph{unique ample model} introduced in \cite{BHJ1}. 

\begin{lem} \label{Lemma unique model}
Any subgeodesic ray $[0,+\infty) \ni t \mapsto \varphi_t \in \mathrm{PSH}(X,\omega) \cap L^{\infty}(X)$ is $L^{\infty}$-compatible with at most one normal, relatively K\"ahler test configuration $(\mathcal{X}, \mathcal{A})$ for $(X,\alpha)$. 
\end{lem}

\begin{proof}
The proof follows the argument outlined above. To this end, consider normal relatively K\"ahler cohomological test configurations $(\mathcal{X}_i, \mathcal{A}_i, \lambda_i, \pi_i)$, $i = 1,2$, for $(X,\alpha)$. By the definition of a cohomological test configuration there are canonical isomorphisms $X \times (\mathbb{P}^1 \setminus \{0\}) \rightarrow \mathcal{X}_i \setminus \pi_i^{-1}(0)$, which induce bimeromorphic maps $\phi_i: X \times \mathbb{P}^1 \dashrightarrow \mathcal{X}_i$, $i = 1,2$.  Note also that we may choose a smooth manifold $\mathcal{Z}$ that \emph{simultaneously} dominates $\mathcal{X}_1$, $\mathcal{X}_2$ and $X \times \mathbb{P}^1$. 
We thus have a resolution of indeterminacy as follows: 

\[
 \begin{tikzcd}[ampersand replacement=\&]
  \&  \mathcal{Z} \arrow[dl, "\mu_1"] \arrow[d,"r"] \ar[dr, "\mu_2"]\\
    \mathcal{X}_1 \& X \times \mathbb{P}^1 \arrow[dashed, l, "\phi_1"] \arrow[dashed,r, "\phi_2"] \& \mathcal{X}_2
 \end{tikzcd}
\]

\noindent Up to replacing $\mathcal{Z}$ by a $\mathcal{Z}'$ we can moreover assume that the dominating morphisms are given by composition of blow-ups with smooth center, hence projective, so that the fibers can be covered by curves.

Now suppose that there is a subgeodesic ray $t \mapsto \varphi_t \in \mathrm{PSH}(X,\omega) \cap L^{\infty}(X)$ that is  $L^{\infty}$-compatible with both $\mathcal{X}_1$ and $\mathcal{X}_2$. In other words, if $\psi_{1}$ and $\psi_{2}$ are Green functions for $D_{1}$ and $D_{2}$ respectively, and $\Phi$ is the $S^1$-invariant function on $X \times \mathbb{P}^1$ corresponding to $(\varphi_t)_{t \geq 0}$, then $r^*\Phi + \psi_{1}$ and $r^*\Phi + \psi_{2}$ are both (locally) bounded on $\mathcal{Z}$. Hence so is their difference, i.e. we have
$$
\psi_{1} - \psi_{2} \in L^{\infty}_{\mathrm{loc}}(\mathcal{Z}).
$$ 
In particular, the Green functions are of the same singularity type, so their respective divisors of singularities coincide, i.e. $[D_1] = [D_2]$. (One way to see this is to consider the associated multiplier ideal sheaves $\mathcal{I}(\psi_{1})$ and $\mathcal{I}(\psi_{2})$. Since $\psi_1$ and $\psi_2$ (locally) differ by a bounded function, we must have $\mathcal{I}(\psi_{1}) = \mathcal{I}(\psi_{2})$. Pulling back via the morphism $r$ we then get $\mathcal{O}_{\mathcal{Z}}(-D_{1}) = r^*\mathcal{I}(\psi_{1}) = r^*\mathcal{I}(\psi_{2}) = \mathcal{O}_{\mathcal{Z}}(-D_{2})$, so that in particular $[D_{2}] = [D_{1}]$ holds).

We now wish to compare the respective pull-backs of $\mathcal{A}_i$ to $\mathcal{Z}$. 
To do this, note that there are $\mathbb{R}$-divisors $D_{1}$ and $D_{2}$ supported on the central fiber $\mathcal{Z}_0$  such that 
$$
\mu_1^*\mathcal{A}_{1} = r^*p_1^*\alpha + [D_{1}],
$$
and 
$$
\mu_2^*\mathcal{A}_{2} = r^*p_1^*\alpha + [D_{2}],
$$
respectively, see \cite[Proposition 3.10]{SD1}. As a consequence, $\mu_1^*\mathcal{A}_{1} = \mu_2^*\mathcal{A}_{2}$ if and only if $[D_{1}] = [D_{2}]$.

Finally, since the cohomology class $\mathcal{A}_{1}$ on $\mathcal{X}_1$ is relatively K\"ahler, there is a K\"ahler form $\eta$ on $\mathbb{P}^1$ such that $\mathcal{A}_i + \pi_i^*\eta_i$ is K\"ahler on $\mathcal{X}_i$. As before, we conclude that the bimeromorphic map $\phi_2 \circ \phi_1^{-1}: \mathcal{X}_1 \dashrightarrow \mathcal{X}_2$ is in fact a morphism: Indeed, let $x \in \mathcal{X}_1$ and $C$ be a curve in $\mu_1^{-1}(x) \subset Z$. Since $\mu_i^*\pi_i^*\eta_i \cdot C = 0$, $i = 1,2$, the projection formula yields \\
$$
0 = (\mu_1^*\mathcal{A}_1 + \mu_1^*\pi_1^*\eta_1) \cdot [C] = (\mu_2^*\mathcal{A}_2 + \mu_1^*\pi_1^*\eta_1) \cdot [C]
$$

$$
= (\mu_2^*\mathcal{A}_2 + \mu_2^*\pi_2^*\eta_2) \cdot [C] = (\mathcal{A}_2 + \pi_2^*\eta_2) \cdot (\mu_2)_*[C].
$$
\\
Hence $\dim \mu_2(C) < 1$. It follows that the bimeromorphic morphism $\mathcal{X}_1 \dashrightarrow \mathcal{X}_2$ is in fact a morphism. By symmetry in $\mathcal{X}_1$ and $\mathcal{X}_2$ we see that if also $\mathcal{A}_{2}$ is relatively K\"ahler, then $\mathcal{X}_1$ and $\mathcal{X}_2$ are isomorphic (and the isomorphism is even given explicitly as the composition $\phi_2 \circ \phi_1^{-1}$).    
\end{proof}

\noindent Note that we may view the above Lemma \ref{Lemma unique model} as an injectivity result. Another useful reformulation of the above is the following:

\begin{cor}
Two test configurations are isomorphic iff their associated geodesic rays are of same singularity type, i.e. if the difference of the associated $S^1$-invariant functions is uniformly bounded.  
\end{cor}

\noindent This is potentially very useful. For instance, we can compute the geodesic rays of cohomological product configurations (see Corollary \ref{Cor product geodesic}). If we can show that the geodesic ray associated to any cohomological test configuration $(\mathcal{X},\mathcal{A})$ with $\mathrm{DF}(\mathcal{X},\mathcal{A}) = 0$ is compatible also with some product configuration $(\mathcal{X}',\mathcal{A}')$ (in fact, coincides with its uniquely associated geodesic), then by Lemma \ref{Lemma unique model} $\mathcal{X} \equiv \mathcal{X}'$, so $(\mathcal{X},\mathcal{A})$ is a product configuration itself. This is a rather general strategy that applies in several situations of interest. 

\begin{rem} A few remarks: 
\begin{itemize}
\item It is possible for a given subgeodesic ray to be compatible with other test configurations that are \emph{not} relatively K\"ahler (e.g. relatively semipositive/nef). Indeed, given a relatively K\"ahler test configuration $(\mathcal{X},\mathcal{A})$ we may use resolution of singularities and pullback to associate a new test configuration $(\hat{\mathcal{X}}, \hat{\mathcal{A}})$, and this is only relatively nef in general. In fact, it is relatively K\"ahler only if $\rho: \hat{\mathcal{X}} \rightarrow \mathcal{X}$ is an isomorphism. However, such modifications of $\mathcal{X}$ do not change the associated geodesic ray, as can be seen e.g. from the techniques for solving the geodesic equation on a normal K\"ahler space (see e.g. \cite{LNM} and references therein). Hence it is not possible to extend this uniqueness result beyond the relatively K\"ahler case.
\item Product test configurations $(X \times \mathbb{P}^1, p_1^*\alpha, \lambda,\pi)$ for $(X,\alpha)$ are automatically normal and relatively K\"ahler (since $\alpha$ is K\"ahler by assumption).  
\end{itemize}
\end{rem}

\medskip

\section{Analytic characterizations of product configurations} \label{Section analytic characterization}

\noindent We here treat the more involved case of cscK manifolds $(X,\omega)$ whose connected automorphism group $\mathrm{Aut}_0(X) \neq \{0\}$ is non-trivial (so $X$ admits holomorphic vector fields). As a main result of this section we give a number of equivalent characterizations of test configurations with vanishing Donaldson-Futaki invariant. 

\medskip

\subsection{Characterizing test configurations with vanishing Donaldson-Futaki invariant} \label{Section characterization vanishing}

As an application of the techniques developed in the previous sections, we establish a number of conditions equivalent to the vanishing of the Donaldson-Futaki invariant. However, for technical reasons this result is formulated for test configurations with vanishing Donaldson-Futaki invariant whose associated geodesic rays are normalized so that $\mathrm{E}(\varphi_t) = 0$ for each $t$. As it turns out, this is not a serious restriction. In order to discuss these points we introduce a certain projection operator on rays and on test configurations: 

Let $(X,\omega)$ be a given cscK manifold with $\alpha := [\omega] \in H^{1,1}(X,\mathbb{R})$ the associated K\"ahler class. Fix $\varphi_0 \in \mathcal{H}_0$ a cscK potential. Consider the projection operator
$$
\mathcal{P}: \mathcal{H} \simeq \mathcal{H}_0 \times \mathbb{R} \longrightarrow \mathcal{H}_0 
$$ 
$$
\varphi \mapsto \varphi - \mathrm{E}(\varphi),
$$
where $\mathrm{E}: \mathcal{H} \rightarrow \mathbb{R}$ is the Aubin-Mabuchi energy functional. If $(\mathcal{X},\mathcal{A})$ is a test configuration for $(X,\alpha)$, then there is a unique geodesic ray $(\varphi_t)_{t \geq 0}$ (emanating from $\varphi_0 \in \mathcal{H}_0$) that is also $L^{\infty}$-compatible with $(\mathcal{X},\mathcal{A})$. The projection $\mathcal{P}$ can also be defined on relatively K\"ahler test configurations, by embedding the latter in the space of geodesic rays emanating from the given cscK potential $\varphi_0 \in \mathcal{H}_0$. This is possible due to the injectivity lemma, see Section \ref{Section unique model}. 

\begin{mydef}
The projection $\mathcal{P}(\mathcal{X},\mathcal{A})$ of a relatively K\"ahler test configuration $(\mathcal{X},\mathcal{A})$ for $(X,\alpha)$ is defined as the unique test configuration $L^{\infty}$-compatible with the geodesic ray $\mathcal{P}(\varphi_t)$.
\end{mydef}

\noindent For later use, and in order to justify the above definition, we establish below some properties (in particular existence) of the projection $\mathcal{P}$ on rays and on test configurations. 

\begin{prop} \label{Prop basic properties of projection} \emph{(Basic properties of the projection)} Let $\mathcal{P}: \mathcal{H} \rightarrow \mathcal{H}_0 \simeq \mathcal{K}$ be as above. Let $(\mathcal{X},\mathcal{A})$ be a normal and relatively K\"ahler test configuration for $(X,\alpha)$. Then
\begin{enumerate}
\item $\mathcal{P}(\mathcal{X},\mathcal{A})$ is a product iff $(\mathcal{X},\mathcal{A})$ is a product. 
\item $\mathrm{J}^{\mathrm{NA}}(\mathcal{P}(\mathcal{X},\mathcal{A})) = \mathrm{J}^{\mathrm{NA}}(\mathcal{X},\mathcal{A})$.
\item $(\mathcal{P}(\varphi_t))_{t\geq 0}$ is a geodesic iff $(\varphi_t)_{t \geq 0}$ is a geodesic.
\item There is a unique normal and relatively K\"ahler cohomological test configuration $\mathcal{P}(\mathcal{X},\mathcal{A})$ for $(X,\alpha)$ with which $\mathcal{P}(\varphi_t)$ is compatible. It equals $(\mathcal{X}, \mathcal{A} - c[\mathcal{X}_0])$, where $c$ is the slope of the linear function $t \mapsto \mathrm{E}(\varphi_t)$. 
\end{enumerate}
\end{prop}

\begin{proof}
The first subpoint $(1)$ is an immediate consequence of $(4)$, which in particular shows that $\mathcal{P}(\mathcal{X}) = \mathcal{X}$. The second point follows from the fact that $\mathrm{J}(\varphi + c) = \mathrm{J}(\varphi)$, $c \in \mathbb{R}$, so in particular $\mathrm{J}(\mathcal{P}(\varphi_t)) = \mathrm{J}(\varphi_t - \mathrm{E}(\varphi_t)) = \mathrm{J}(\varphi_t)$ for each $t \in [0,+\infty)$. Dividing by $t$ and passing to the limit, assertion $(2)$ follows.
The subpoint $(3)$ follows since $\mathcal{P}(\varphi_t) = \varphi_t - \mathrm{E}(\varphi_t) = \varphi_t + at + b$ for some $a,b \in \mathbb{R}$, since the function $t \mapsto \mathrm{E}(\varphi_t)$ is linear along geodesics. It follows that $\mathrm{E}(\varphi_t)$ is harmonic for each $t$. Assertion $(3)$ follows. Finally, in order to prove the last subpoint, recall that compatibility is determined by $\mathcal{X}$ and the Green function $\psi_D$ associated to the divisor $D$ supported on $\mathcal{X}_0$, satisfying $\rho^*\mathcal{A} = \mu^*p_1^*\alpha + \delta_D$ (see figure).

\[\begin{tikzcd} \mathcal{X'} \arrow[swap]{d}{\rho}  \arrow{dr}{\mu} \\ \mathcal{X} \arrow[r, dashed] & X \times \mathbb{P}^1 \end{tikzcd} \]

\noindent But changing $\varphi_t$ for $\mathcal{P}(\varphi_t)$ we preserve the compatibility relation (see Definition \ref{Definition compatibility}) by also changing $\psi_D$ for $\psi_D - at - b$, where $a$ is the slope of the linear function $t \mapsto \mathrm{E}(\varphi_t)$. This corresponds precisely to the test configuration $(\mathcal{X}, \mathcal{A} - c[\mathcal{X}_0])$, which is relatively K\"ahler iff $(\mathcal{X},\mathcal{A})$ is.  
\end{proof}

\noindent (It is immediately clear that $\mathcal{P}(\mathcal{X},\mathcal{A})$ is normal and relatively K\"ahler, because $\mathcal{X}_0 = \pi^{-1}(0)$ is just a single fiber).

\begin{rem} \label{Cor product geodesic}
Note that the geodesic ray associated to a product test configuration $\mathcal{P}(\mathcal{X},\mathcal{A})$ induced by a one-parameter subgroup $\lambda: \mathbb{C}^* \rightarrow \mathrm{Aut}(X)$ \emph{(cf. Example \ref{Example product})} is given by $\mathcal{P}(\varphi_t) = \lambda(\tau). \varphi_0$ for each $t$. The non-normalized geodesic ray associated to $(\mathcal{X},\mathcal{A})$ still satisfies $\lambda(\tau)^*\omega = \omega_{\varphi_t}$. One may note that these rays are precisely the ones studied already by Mabuchi in \cite{MabuchiKenergy1}. 
\end{rem}

\medskip

\subsection{Proof of Theorem \ref{Main theorem analytic product conditions}}

We now set out to prove the following analog of Theorem \ref{Theorem equivalent characterizations discrete} in the case of $\mathrm{Aut}_0(X) \neq \{0\}$ (cf. also \cite[Lemma 3.1]{BDL} in the polarized case). It holds only for test configurations whose geodesic rays are normalized so that $\mathrm{E}(\varphi_t) = 0$. Otherwise put, given a test configuration $(\mathcal{X},\mathcal{A})$ we may instead consider $\mathcal{P}(\mathcal{X},\mathcal{A})$ which then satisfies this property. Since the total space does not change under projection, this is often not a serious restriction (e.g. if one wishes to prove that test configurations with vanishing Donaldson-Futaki invariant are products, see Proposition \ref{Prop basic properties of projection}).

\begin{thm} \label{Theorem analytic product conditions} \emph{(cf. Theorem \ref{Main theorem analytic product conditions})}
Suppose that $(X,\omega)$ is a cscK manifold, with $\alpha := [\omega] \in H^{1,1}(X,\mathbb{R})$ the corresponding K\"ahler class. Let $(\mathcal{X},\mathcal{A})$ be a normal and relatively K\"ahler test configuration for $(X,\alpha)$ whose associated geodesic ray $(\varphi_t)_{t \geq 0}$ satisfies $\mathrm{E}(\varphi_t) = 0$ for each $t \in [0, +\infty)$. Let $J: TX \rightarrow TX$ be the complex structure and $\omega$ a cscK metric on $X$. Then the following statements are equivalent:
\begin{enumerate}
\item $\mathrm{DF}(\mathcal{X},\mathcal{A}) = 0$. 
\item The central fiber $\mathcal{X}_0$ is reduced and $\mathrm{M}^{\mathrm{NA}}(\mathcal{X},\mathcal{A}) = 0$. 
\item The central fiber $\mathcal{X}_0$ is reduced and the Mabuchi K-energy functional is constant along the geodesic ray $(\varphi_t)_{t \geq 0}$ associated to $(\mathcal{X},\mathcal{A})$, i.e. we have $\mathrm{M}(\varphi_t) = \mathrm{M}(\varphi_0)$ for each $t \in [0,+\infty)$.

\item The central fiber $\mathcal{X}_0$ is reduced and the associated geodesic ray satisfies
$$
\inf_{g \in G} \mathrm{J}(g. \varphi_t) = 0 \;\; \mathrm{and} \; \; \inf_{g \in G} d_1(0, g.\varphi_t)  = 0.
$$

\item The central fiber $\mathcal{X}_0$ is reduced and there is a real holomorphic Hamiltonian vector field $V$ such that the geodesic ray $(\varphi_t)_{t \geq 0}$ associated to $(\mathcal{X},\mathcal{A})$ satisfies $\mathrm{exp}(tV)^*\omega = \omega$ and $\mathrm{exp}(tJV)^*\omega = \omega_{\varphi_t}$.
\item The central fiber $\mathcal{X}_0$ is reduced and the associated geodesic ray $(\varphi_t)$ consists entirely of cscK potentials. More precisely, if $\bar{\mathcal{S}}$ denotes the mean scalar curvature of $\omega_{\varphi_0}$, then
$$
\mathrm{\mathcal{S}(\omega_{\varphi_t}) = \bar{\mathcal{S}}}
$$
for each $t \in [0, +\infty)$. 
\end{enumerate} 
\end{thm}

\noindent The structure of the proof is the following: The implications $(1) \Rightarrow (2) \Rightarrow (5)$ rely on the result \cite[Theorem C]{SD1} on asymptotics of the Mabuchi functional, extended to the setting of possibly singular cohomological test configurations. It is worth pointing out that, unlike in the polarized case, the asymptotics a priori yields only an upper bound $o(t)$ in $(4)$. However, one can then improve this by other means 
to find the given statement above, namely by first proving $(5)$ and noting that $(5) \Rightarrow (4)$. In order to establish $(4) \Leftrightarrow (5)$ we also rely on a very slight variation of the proof of \cite[Lemma 3.1]{BDL}, using the important $G$-coercivity result \cite[Theorem 1.5]{BDL} and noting that $\mathrm{J}_G$ and $d_{1,G}$ are comparable \cite[Lemma 5.11]{DR}. Further, we show $(5) \Rightarrow (2) \Rightarrow (1)$ by means of simple argument expressing the slope of the Mabuchi functional along the rays given by the flow of a holomorphic Hamiltonian vector field as in $(5)$ in terms of the (original) Futaki invariant introduced in \cite{Futaki}. Several of these arguments seem to be closely related to the seminal work of Mabuchi \cite[Section 5]{Mabuchi}. 

\begin{rem} The equivalence $(1) \Leftrightarrow (2)$ in fact holds without the assumption that $\mathrm{E}(\varphi_t) = 0$ for each $t$. 
\end{rem}

\subsubsection{Proof of $(1) \Rightarrow (2) \Rightarrow (5) $}

As a first application of the formalism developed in Section \ref{Subsection associated rays} we show that normal and relatively K\"ahler test configurations $(\mathcal{X},\mathcal{A})$ with vanishing Donaldson-Futaki invariant automatically have reduced central fiber, and the $d_{1,G}$-length of geodesics associated to such test configurations grows like a small $o(t)$.  Recall that we view $\varphi_t(x) := \varphi(t,x)$ as a function on $X \times \mathbb{P}^1 \setminus \{0\}$, and $G := \mathrm{Aut}(X)_0$.

\begin{lem} \label{Lemma JNA bounded}
Suppose that $(\mathcal{X}, \mathcal{A})$ is a normal, relatively K\"ahler cohomological test configuration for $(X,\alpha)$. Let $[0, +\infty) \ni t \mapsto \psi_t \in \mathrm{PSH}(X,\omega) \cap C^{\infty}(X)$ be as in Theorem \ref{Theorem generalized thm C}, emanating from a cscK potential $\psi_0 \in \mathcal{H}_0$. If $\mathrm{DF}(\mathcal{X},\mathcal{A}) = 0$, then
\begin{enumerate}
\item $\mathcal{X}_{0,red} = \mathcal{X}_0$. 
\item 
$
0 \leq d_{1,G}(G\psi_0, G\psi_t) \leq o(t). 
$
\end{enumerate}
\end{lem}

\begin{proof} If $\mathrm{DF}(\mathcal{X}, \mathcal{A}) = 0$ it follows from Theorem \ref{Theorem generalized thm C} that 
$$
0 \leq \lim_{t \rightarrow +\infty} \frac{{\mathrm{M}}(\psi_t)}{t} = 0 + ((\mathcal{X}_{0,red} - \mathcal{X}_0)\cdot\mathcal{A}^n)\leq 0, 
$$
where the lower bound of ${\mathrm{M}}$ holds because $\varphi_0$ is a cscK potential. This forces $\mathcal{X}_{0,red} = \mathcal{X}_0$.  As a consequence, we have
$$
\lim_{t \rightarrow +\infty} \frac{{\mathrm{M}}(\psi_t)}{t} \leq 0,
$$ 
so ${\mathrm{M}}(\psi_t) \leq o(t)$. Since $(X,\alpha)$ is a cscK manifold, the Mabuchi functional is $G$-coercive \cite[Theorem 1.1]{BDL}. Hence $0 \leq \mathrm{J}_G(G\psi_t) \leq o(t)$. 
The conclusion $(2)$ then follows immediately from the fact that the growth of $\mathrm{J}$-functional is the same as that of the $d_1$-metric \cite[Proposition 5.5]{DR}.  
\end{proof}

\begin{rem}
The above also holds if one replaces the hypothesis $\mathrm{DF}(\mathcal{X},\mathcal{A}) = 0$ with ${{\mathrm{M}}^{\mathrm{NA}}}(\mathcal{X},\mathcal{A}) = 0$.
\end{rem}

\noindent Lemma \ref{Lemma JNA bounded} is useful in combination with the following observation:

\begin{lem} \label{Lemma change of rays} Suppose that $|\psi_t - \varphi_t|$ is uniformly bounded in $t$. Then $\mathrm{J}_G(G\varphi_t) \leq o(t)$ iff $\mathrm{J}_G(G\psi_t) \leq o(t)$. 
\end{lem} 

\begin{proof}
Suppose that $\varphi_t - \psi_t$ is uniformly bounded in $t$. By \cite[Proposition 2.2]{BDL} the infimum
$
\mathrm{J}_G(G\psi_t) = \inf_{g \in G} \mathrm{J}(g.\psi_t)
$ 
is attained, so there is a constant $C > 0$ and a sequence $\{g_t\}_{t \geq 0} \subset G^{\mathbb{N}}$ such that
$
\mathrm{J}_G(G\psi_t) = \int_X g_t.\psi_t \; \omega^n \leq C. 
$
Since $0 \leq \mathrm{J}_G(G\psi_t) \leq o(t)$ by assumption, by possible increasing the constant $C$, the expression \eqref{equation action} yields 
$$
0 \leq \mathrm{J}_G(G\varphi_t)  \leq \mathrm{J}_G(G\psi_t) + \left\{ \int_X g_t.\varphi_t \; \omega^n - \int_X g_t.\psi_t \; \omega^n \right\}
$$
$$
= \mathrm{J}_G(G\psi_t) + \int_X g_t^*(\varphi_t - \psi_t) \; \omega^n \leq o(t) + C = o(t).
$$
Indeed, the function $g_t^*(\varphi_t - \psi_t)$ is uniformly bounded in $t$, because $\varphi_t - \psi_t$ is. 
\end{proof}

\noindent Note that any smooth subgeodesic ray $t \mapsto \psi_t$, $C^{\infty}$-compatible with a given test configuration $(\mathcal{X},\mathcal{A})$ is also $L^{\infty}$-compatible with the unique geodesic ray $t \mapsto \varphi_t$ associated to $(\mathcal{X},\mathcal{A})$. As a consequence of the above Lemma \ref{Lemma change of rays} we then see that the conclusion of Lemma \ref{Lemma JNA bounded} holds for the unique geodesic ray $t \mapsto (\varphi_t)_{t \geq 0}$ associated to the test configuration $(\mathcal{X},\mathcal{A})$, see \cite[Lemma 4.6]{SD1} for the construction: 

\begin{cor} \label{corollary 3.13}
Suppose that $(\mathcal{X}, \mathcal{A})$ is a normal, relatively K\"ahler cohomological test configuration for $(X,\alpha)$ satisfying $\mathrm{DF}(\mathcal{X},\mathcal{A})$. Let $[0, +\infty) \ni t \mapsto \varphi_t \in \mathrm{PSH}(X,\omega) \cap C^{\infty}(X)$ be the unique associated geodesic ray emanating from a cscK potential $\varphi_0 \in \mathcal{H}_0$. Then $\mathcal{X}_{0,red} = \mathcal{X}_0$ and 
$
0 \leq d_{1,G}(G\varphi_0, G\varphi_t) \leq o(t). 
$

\end{cor}

\noindent The following result is a very slight modification of the compactness argument of \cite[Propositon 3.1]{BDL}. The only new observation is that we may replace the upper bound $\leq C$ by $\leq o(k)$. We give the necessary details showing that the proof then goes through in the same way. Note that the argument in question crucially relies on the assumption that the ray is a geodesic (cf. Proposition \ref{Prop basic properties of projection}). See Section \ref{Section G-action} for the definition of the action on potentials. 

\begin{prop} \label{Prop existence of V} \emph{(cf. \cite[Proposition 3.1]{BDL})}
Suppose that $(\mathcal{X},\mathcal{A})$ is a normal and relatively K\"ahler test configuration for $(X,\alpha)$, with $\mathrm{DF}(\mathcal{X},\mathcal{A}) = 0$. Let $\varphi_0 \in \mathcal{H}_0$ be a cscK potential and suppose that the unique associated geodesic ray $[0, +\infty) \ni t \mapsto \varphi_t \in \mathrm{PSH}(X,\omega) \cap L^{\infty}(X)$ emanating from $\varphi_0$ is normalized so that $\mathrm{E}(\varphi_t) = 0$ for each $t$. Then there is a real holomorphic Hamiltonian vector field $V \in \mathrm{isom}(X,\omega_{\varphi_0})$ such that 
$
\varphi_t = \mathrm{exp}(tJV).\varphi_0.
$ 

\end{prop}

\begin{proof}
By Corollary \ref{corollary 3.13} the $d_{1,G}$-length of the geodesic $(\varphi_t)$ is controlled by the inequality $0 \leq d_{1,G}(G\varphi_0, G\varphi_t) \leq o(t)$. We may thus find a sequence $\{g_k\}_{k \in \mathbb{N}} \subset G^{\mathbb{N}}$ such that 
$
k^{-1} d_1(\varphi_0, g_k \cdot \varphi_k) \rightarrow 0
$
as $k \rightarrow +\infty$.
Because $G$ is reductive we may for each $k$ write $g_k = h_k \mathrm{exp}(-JV_k)$, where $h_k \in \mathrm{Isom}_0(X,\omega_{\varphi_0})$ and $V_k \in \mathrm{isom}(X,\omega_{\varphi_0})$ is a non-zero real holomorphic Hamiltonian vector field , see e.g. \cite[Propositions 6.2 and 6.9]{DR}. Since $G$ acts on $\mathcal{H}$ by $d_1$-isometries \cite[Lemma 5.9]{DR} we see that
\begin{equation} \label{equation 4144}
o(k) \geq d_1(\varphi_0, g_k\cdot \varphi_k) = d_1(g_k^{-1} \cdot \varphi_0, \varphi_k) = d_1(\mathrm{exp}(JV_k)\cdot \varphi_0, \varphi_k).
\end{equation}
The geodesic $t \mapsto \varphi_t$ has constant speed which we may assume equal to $t$, i.e. $d_1(\varphi_0, \varphi_t) = t$, so the triangle inequality yields the double inequality
\begin{equation} \label{equation double estimate}
k - o(k) \leq d_1(\varphi_0, \mathrm{exp}(JV_k) \cdot \varphi_0) \leq k + o(k).
\end{equation}
Now note that $t \mapsto \mathrm{exp}(tJV).\varphi_0$ is a $d_1$-geodesic ray emanating from $\varphi_0$ (see \cite[Section 7.2]{DR}). Since $d_1$ is linear along geodesics it follows from \eqref{equation double estimate} that 
$$
1 - \frac{o(k)}{k} \leq  d_1(\varphi_0, \mathrm{exp}\left(\frac{JV_k}{k}\right).\varphi_0) \leq 1 + \frac{o(k)}{k}.
$$
Using \eqref{equation action} we conclude that there is a $D > 1$ and a $k_0 \in \mathbb{N}$ such that $1/D \leq ||JV_k/k|| \leq D$ for all $k \geq k_0$.
Since the space of holomorphic Hamiltonian Killing vector fields of $(X,\omega_{\varphi_0},J)$ is finite dimensional, it follows that there is a subsequential limit, i.e. there is a  $0 \neq V \in \mathrm{isom}(X,\omega_{\varphi_0})$ such that $V_{k_j}/k_j \rightarrow V$ as $k_j \rightarrow +\infty$. 
We finally argue that in fact $\varphi_t = \mathrm{exp}(tJV)\cdot \varphi_0$, following \cite{BDL}: For each $k$, consider the smooth $d_1$-geodesic segments
$$
[0,k] \ni t \mapsto \varphi_t^k := \mathrm{exp}\left(t\frac{V_k}{k}\right) \cdot \varphi_0 \in \mathcal{H}_0
$$
Note that the function $t \mapsto h(t) := d_1(\varphi_t^k, \varphi_t)$ is convex (since $(\varphi_t)$ and $(\varphi_t^k)$ are both $d_1$-geodesic rays), with $h(0) = 0$ and $h(k) \leq o(k)$. 
The above inequality \eqref{equation 4144} then implies that, for any fixed $t \in [0,k]$ we have   
$
d_1(\varphi_t^k, \varphi_t) \leq t\frac{h(k)}{k} \rightarrow 0,
$
as $k \rightarrow +\infty$. Exactly as in \cite[Lemma 2.7]{BDL} we then get that $\mathrm{exp}(t{V_{k_j}}/{k_j}) \cdot \varphi_0 \longrightarrow \mathrm{exp}(tJV)\cdot \varphi_0$ smoothly, concluding the proof. 
\end{proof}

\noindent Putting these results together we see that if $\mathrm{DF}(\mathcal{X},\mathcal{A}) = 0$, then $\mathcal{X}_0$ is reduced and so $((\mathcal{X}_{0,red} - \mathcal{X}_0) \cdot \mathcal{A}^n) = 0$. In particular $\mathrm{M}^{\mathrm{NA}}(\mathcal{X},\mathcal{A}) = \mathrm{DF}(\mathcal{X},\mathcal{A}) = 0$ as well, so that $(1) \Rightarrow (2)$. Furthermore, the above arguments show that $(2)$ implies $d_{1,G}(\varphi_0, G.\varphi_t) \leq o(t)$ and so in turn by Proposition \ref{Prop existence of V} the statement $(5)$ follows.

\subsubsection{Proof of $(5) \Leftrightarrow (4)$}

The implication $(5) \Rightarrow (4)$ is immediate by definition, since then $$0 \leq \mathrm{J}_G(G \varphi_t) \leq \mathrm{J}(G\varphi_0) = 0$$ for each $t$. 
The converse implication $(4) \Rightarrow (5)$ follows precisely as in Proposition \ref{Prop existence of V} above.

\subsubsection{Proof of $(5) \Rightarrow (2) \Rightarrow (1)$}

\begin{lem} \label{Lemma analytic product condition}
Suppose that $h_{\omega}^V$ is a Hamiltonian potential for the real holomorphic vector field $V \in \mathfrak{isom}(X,\omega)$ with respect to $\omega$, i.e. $i_V(\omega) = i\bar{\partial}h_{\omega}^V$, where $h_{\omega}^V \in C^{\infty}(X,\mathbb{R})$. Then for all $t \geq 0$ we have $\mathrm{exp}(tV)^*\omega = \omega$ and $\mathrm{exp}(tJV)^*\omega = \omega_{\psi_t }$, 
where $(\psi_t)_{t \geq 0}$ is a smooth  ray that can be chosen so that
$
\dot{\psi}_t = \mathrm{exp}(tJV)^*h_{\omega}^V  
$
and $\mathrm{E}(\psi_t) = 0$ for each $t$.
\end{lem}
\begin{proof}
Write $f_t$ for the flow $\mathrm{exp}(tJV)$ of $JV$. Then $i_{V}(\omega) = \sqrt{-1}\bar{\partial}h_{\omega}^V$ implies that 
$$
i_{V}(f_t^*h_{\omega}^V) = f_t^*i_{V}(\omega) = f_t^*(\sqrt{-1}\bar{\partial}h_{\omega}^V) = \sqrt{-1}\bar{\partial}(f_t^*h_{\omega}^V). 
$$
Hence
$$
d \circ i_{V}(f_t^*\omega) = \sqrt{-1}\partial\bar{\partial}f_t^*h_{\omega}^V = dd^cf_t^*h_{\omega}^V.
$$
On the other hand there is a unique smooth ray 
$(\psi_t)_{t \geq 0}$ in $\mathcal{H}_0$ such that for all $t$ we have $\mathrm{E}(\psi_t) = 0$ and $f_t^*\omega = \omega_{\psi_t}$. As a consequence
$$
d \circ i_{V}(f_t^*\omega) = d \circ i_{V}(\omega_{\psi_t})  = L_{V}(\omega_{\psi_t}) = dd^c\dot{\psi}_t.  
$$
Since $X$ is compact we then have $\dot{\psi}_t = f_t^*h_{\omega}^V + C$, for some constant $C = C(t)$. We then conclude by showing that with the normalization
$$
\int_X h_{\omega}^V \omega^n = 0, 
$$
since we assume that $\mathrm{E}(\psi_t) = 0$ for each $t$, we have
$$
0 = \frac{d}{dt}\mathrm{E}(\psi_t) = \int_X \dot{\psi}_t \omega_{\psi_t}^n = \int_X h_{\omega}^V \omega^n + \int_X C(t) \omega_{\psi_t}^n = \int_X C(t) \omega^n,
$$
and hence $C(t) = 0$ for each $t$. 
\end{proof}

\begin{prop} \label{Prop Mabuchi constant}
Let $(\varphi_t)_{t \geq 0}$ be the unique geodesic ray associated to $(\mathcal{X},\mathcal{A})$. Then $t \mapsto \mathrm{M}(\varphi_t)$ is linear with slope given by the Futaki invariant $\mathrm{Fut}_{\alpha}(V)$. In particular, if $\varphi_0 \in \mathcal{H}_0$ is a cscK potential, then
$
\mathrm{M}(\varphi_t) = \mathrm{M}(\psi_0)
$
for all $t \geq 0$. If moreover the central fiber $ \mathcal{X}_0 $ is reduced, then $\mathrm{DF}(\mathcal{X},\mathcal{A}) = 0$.
\end{prop}
\begin{proof}
Since the geodesic ray $(\varphi_t)_{t \geq 0}$ is smooth we may compute
$$
V  \cdot \frac{d}{dt}\mathrm{M}(\varphi_t) = \int_X \dot{\varphi}_t (\bar{S} - S(\omega_{\varphi_t})) \omega_{\varphi_t}^n
$$ 
$$
= \int_X f_t^*h_{\omega}^V (\bar{S} - f_t^*S(\omega)) f_t^*\omega^n = \int_X h_{\omega}^V (\bar{S} - S(\omega)) \omega^n = V \cdot \mathrm{Fut}_{\alpha}(V). 
$$
We have here used that $f_t^*S(\omega) = S(f_t^*\omega)$. If we assume that $S(\omega) = \bar{S}$, i.e. that $0 \in \mathcal{H}$ is a cscK potential, then $\mathrm{Fut}_{\alpha}(V)$ vanishes. 
By linearity $t \mapsto \mathrm{M}(\varphi_t)$ is constant and $\mathrm{Fut}_{\alpha}(V) = \mathrm{M}^{\mathrm{NA}}(\mathcal{X},\mathcal{A}) = 0$. If the central fiber $\mathcal{X}_0$ is reduced we moreover have $\mathrm{DF}(\mathcal{X},\mathcal{A}) = \mathrm{M}^{\mathrm{NA}}(\mathcal{X},\mathcal{A}) = 0$, which is what we wanted to prove.
\end{proof}

\noindent Putting this together shows that $(5) \Rightarrow (2) \Rightarrow (1)$. 

\subsubsection{Proof of $(2) \Leftrightarrow (3)$}

We here wish to show that if $\mathrm{M}^{\mathrm{NA}}(\mathcal{X},\mathcal{A}) = 0$ then $\mathrm{M}$ is constant along the associated geodesic ray emanating from the given cscK potential $\varphi_0 \in \mathcal{H}_0$. The proof is given precisely as in Proposition \ref{Prop Mabuchi constant}. Indeed, if $\mathrm{M}^{\mathrm{NA}}(\mathcal{X},\mathcal{A}) = 0$ and the central fiber $\mathcal{X}_0$ is reduced, then $(5)$ holds, so in particular we are in the situation of Proposition \ref{Prop Mabuchi constant}. As before
$$
0 = \frac{d}{dt}\mathrm{M}(\varphi_t) = \mathrm{Fut}_{\alpha}(V). 
$$
and so $\mathrm{M}(\varphi_t)$ is constant. The converse is immediate, since by Theorem \cite[Theorem C]{SD1} we have 
$$
\mathrm{M}^{\mathrm{NA}}(\mathcal{X},\mathcal{A}) = \lim_{t \rightarrow +\infty} \frac{\mathrm{M}(\varphi_t)}{t}.
$$ 

\subsubsection{Proof of $(3) \Leftrightarrow (6)$}

The equivalence $(3) \Leftrightarrow (6)$ is an immediate consequence of the characterization of cscK metrics as the minima of the Mabuchi functional: Indeed, one the one hand, if $\mathrm{M}(\varphi_t) = \mathrm{M}(\varphi_0)$ then since $\varphi_0$ is a minimum, so is $\varphi_t$. Hence $\varphi_t$ is a cscK metric. 
Conversely, the Mabuchi functional is (tautologically) constant on the minimum set.

\noindent We now discuss some immediate applications in the special case when the automorphism group is discrete.

\medskip

\subsection{The case of $\mathrm{Aut}(X)$ finite}

In case the automorphism group is discrete the geodesic K-polystability notion is in fact equivalent to K-polystability. The argument also simplifies, leading to a direct proof of K-stability that does not rely on the machinery of lifting of cscK classes to blowups (due to Arezzo-Pacard \cite{ArezzoPacard,ArezzoPacardII} and others, see \cite{DervanRoss} building on arguments of Stoppa \cite{StoppacscKmetricsimpliesstability}). 


\begin{thm} \label{Theorem equivalent characterizations discrete}
Suppose that $(X,\omega)$ is a cscK manifold with discrete automorphism group. Let $\alpha := [\omega] \in H^{1,1}(X,\mathbb{R})$ be the corresponding K\"ahler class. For any cohomological test configuration $(\mathcal{X},\mathcal{A})$ for $(X,\alpha)$ the following are equivalent: 
\begin{enumerate}
\item $\mathrm{DF}(\mathcal{X},\mathcal{A}) = 0$
\item $\mathrm{J}^{\mathrm{NA}}(\mathcal{X},\mathcal{A}) = 0$
\item $(\mathcal{X},\mathcal{A})$ is the trivial test configuration. 
\end{enumerate}
\end{thm}

\noindent It is worth pointing out that the above results entail that uniform K-stability implies K-stability, as the terminology suggests: 

\begin{cor}
If $\alpha \in \mathcal{C}_X$ is uniformly K-stable then it is K-stable. 
\end{cor}

\noindent This gives one way of showing that cscK manifolds are K-stable, since in fact proving uniform K-stability using the energy functional asymptotics developed in \cite{SD1} is not much more difficult than proving K-semistability. 
In this setting we can also prove that geodesic K-stability is equivalent to K-stability in the usual sense, thus also yielding a new proof of K-stability of cscK manifolds (cf. \cite{DervanRoss}: 

\begin{cor}
Suppose that $(X,\omega)$ is a cscK manifold admitting no non-trivial holomorphic vector fields. Then $(X,[\omega])$ is K-stable. 
\end{cor}

\noindent For other proofs in the K\"ahler setting, see \cite{DR, SD1}. We finally discuss product configurations and geodesic K-polystability. 

\medskip

\subsection{Product configurations and geodesic K-polystability} \label{Section geodesic kps}

In the case of polarized manifolds it is well-known that the geodesic ray associated to a product test configuration is given by pullback via the $\mathbb{C}^*$-action (see e.g. \cite{Berman}). Moreover, up to a multiplicative constant the Donaldson-Futaki invariant equals the Futaki invariant. In particular, the Donaldson-Futaki invariant of a product configuration for a cscK manifold (or merely K-semistable) $(X,\alpha)$ must vanish. We here check the analog of these results for cohomological product configurations:   

\begin{prop} \label{Prop product vanish}
Let $(\mathcal{X},\mathcal{A})$ be a product test configuration for any \emph{(not necessarily cscK)} pair $(X,\alpha)$, induced by a $\mathbb{C}^*$-action $\lambda$ on $X$. Let $V_{\lambda}$ be the infinitesimal generator of $\rho$. Then the geodesic ray associated to $(\mathcal{X},\mathcal{A})$ is of the form $\varphi_t = \exp(tJV_{\lambda}).\varphi_0$. 
\end{prop}

\begin{proof}
Any product test configuration is given by $X \times \mathbb{P}^1$ with the $\mathbb{C}^*$-action given as the diagonal action induced by a one-parameter subgroup $\lambda: \mathbb{C}^* \rightarrow \mathrm{Aut}(X)$, i.e. $\tau \cdot (x,z) := (\lambda(\tau)\cdot x, \tau z)$. The canonical morphism associated to the induced product configuration is given by $\mu: X \times \mathbb{P}^1 \rightarrow X \times \mathbb{P}^1$ is given by $(x,z) \mapsto (\lambda^{-1}(z) \cdot x, z)$. 
We now compute the geodesic ray associated to $(\mathcal{X},\mathcal{A})$: In the notation of \cite[Lemma 4.6]{SD1}, let $M := \pi^{-1}(\bar{\Delta}) \subset \mathcal{X} $. It is a smooth complex manifold with boundary $\partial M = \pi^{-1}(S^1)$. Let $D$, $\theta_D$, $\psi_D$ and $\Omega$ be as in Section \ref{Section energy functional asymptotics}. Since $\Omega$ is relatively K\"ahler there is an $\eta \in H^{1,1}(\mathbb{P}^1)$ such that $\Omega + \pi^*\eta$ is K\"ahler on $\mathcal{X}$. We may then write $\tilde{\Omega} = \Omega + \pi^*\eta + dd^cg$,
where $\tilde{\Omega}$ is a K\"ahler form on $\mathcal{X}$ and  $g \in \mathcal{C}^{\infty}(\mathcal{X})$. In a neighbourhood  of $\bar{\Delta}$ the form $\eta$ is further $dd^c$-exact, and so we write $\eta = dd^c (g' \circ \pi)$ for a smooth function $g' \circ \pi$ on $\bar{\Delta}$. In order to construct the weak geodesic ray associated to $(\mathcal{X},\mathcal{A})$ we consider the following homogeneous complex Monge-Amp\'ere equation 
\[
    (\star)  \left\{
                \begin{array}{ll}
                  (\tilde{\Omega} + dd^c\tilde{\Psi})^{n+1} = 0 \; \; \mathrm{on}\; \mathrm{Int}(M) \\
                  \tilde{\Psi}_{\vert \partial M} = \varphi_0 + \psi_D - g' -g 
                \end{array}
              \right.
  \] 
\noindent The solution is given by the Perron envelope
$$
\tilde{\Psi} = \sup \{ \xi \in \mathrm{PSH}(M,\Omega): \xi_{\vert \partial M} \leq \varphi_0 + \psi_D - g' -g \}
$$
which simply equals the constant function $\phi_0 + \psi_D - g' -g$ here. As explained in the construction (see \cite[Lemma 4.6]{SD1}) the weak geodesic ray $(\varphi_t)_{t \geq 0}$ is defined by asking that $S^1$-invariant function $\Phi(x,e^{-t+is}) = \varphi_t(x)$ on $X \times \Delta^*$ satisfies 
$
\mu^*(p_1^*\omega + dd^c\Phi) = \tilde{\Omega} + dd^c\tilde{\Phi}. 
$
Since the latter form identifies with $p_1^*\omega_{\varphi_0}$ we see from the above that $\lambda(\tau)^*\omega_{\varphi_0} = \omega_{\varphi_t}$. In particular, if the geodesic ray is normalized so that $\mathrm{E}(\varphi_t) = 0$ for each $t$, then by definition of the action of $G := \mathrm{Aut}_0(X)$ on the space $\mathcal{H} \cap E^{-1}(0)$ of normalized K\"ahler potentials (cf. Section \ref{Section G-action}) it follows that 
$
\lambda(\tau)^*\omega_{\varphi_0} = \omega + dd^c\lambda(\tau).\varphi_0.  
$ 
This finishes the proof.
\end{proof}

\begin{cor}
Let $(\mathcal{X},\mathcal{A})$ be a product test configuration for any \emph{(not necessarily cscK)} pair $(X,\alpha)$, induced by a $\mathbb{C}^*$-action $\lambda$ on $X$. Let $V_{\lambda}$ be the infinitesimal generator of $\rho$. Then 
$$
\mathrm{DF}(\mathcal{X},\mathcal{A}) = \mathrm{Fut}_{\alpha}(V_{\lambda}).
$$ 
In particular, if $(X,\alpha)$ is K-semistable, then $\mathrm{Fut}_{\alpha}(V) = 0$ for each $V \in \mathfrak{h}$. 
\end{cor}

\begin{proof}
It then follows from Theorem \cite[Theorem C]{SD1} and a calculation identical to that in the proof of Proposition \ref{Prop Mabuchi constant} that 
$$
\mathrm{DF}(\mathcal{X},\mathcal{A}) = \lim_{t \rightarrow +\infty} \frac{\mathrm{M}(\varphi_t)}{t} = \mathrm{Fut}_{\alpha}(V_{\lambda}).
$$  
In particular, if $(X,\alpha)$ is cscK then $\mathrm{DF}(\mathcal{X},\mathcal{A}) = 0$. 
\end{proof}

\noindent This motivates the definition of \emph{geodesic K-polystability}, which is natural from an analytic point of view:  

\begin{mydef} \label{Definition geodesic Kps}
The pair $(X,\alpha)$ is said to be \emph{geodesically K-polystable} if $\mathrm{DF}(\mathcal{X},\mathcal{A}) \geq 0$ for all relatively K\"ahler test configurations $(\mathcal{X},\mathcal{A})$, with equality if and only if the geodesic ray associated to $(\mathcal{X},\mathcal{A})$ is of the form $\varphi_t = \exp(tJV).\varphi_0$ for some real holomorphic Hamiltonian vector field $V \in \mathfrak{h}$. 
\end{mydef}

\noindent Indeed, if $(X,\alpha)$ is K-semistable then 
it follows from the above that test configurations with vanishing Donaldson-Futaki invariant are \emph{precisely} of this form $\varphi_t = \exp(tJV).\varphi_0$ for some $V \in \mathfrak{h}$. In general, the geodesic K-polystability notion compares to K-polystability for K\"ahler manifolds (Definition \ref{Def Kps}) as follows: 



\begin{prop}
If $(X,\alpha)$ is K-polystable then it is geodesically K-polystable. 
\end{prop}

\begin{proof}
This is an immediate consequence of Proposition \ref{Prop product vanish}. 
\end{proof}

\noindent 
\noindent In case the automorphism group is discrete these notion are in fact equivalent, and we expect this to hold in general. Interestingly, it can however be proven that geodesic K-polystability is \emph{stronger} than the classical algebraic K-polystability for polarized manifolds due to Donaldson \cite{Donaldsontoric} 

\begin{prop} \label{Prop comparison polarized} 
Let $(X,\omega)$ be a compact K\"ahler manifold with $\alpha := [\omega] \in H^{1,1}(X,\mathbb{R})$ the associated K\"ahler class. 
\begin{enumerate}
\item If $\alpha = c_1(L)$ for some ample line bundle $L$ over $X$, then geodesic K-polystability implies algebraic K-polystability in the classical sense of Donaldson \cite{Donaldsontoric}.
\item If $\mathrm{Aut}(X)$ is discrete, then $(X,\alpha)$ is geodesically K-polystable if and only if $(X,\alpha)$ is K-polystable in the sense of \cite{SD1}. 
\end{enumerate}
\end{prop}

\begin{proof}
First note that K-semistability is already known for arbitrary K\"ahler manifolds $(X,\alpha)$ such that $\alpha$ admits a cscK representative, see \cite[Theorem A]{SD1}.
In order to prove $(1)$, suppose that $(\mathcal{X},\mathcal{L})$ is a relatively ample test configuration for $(X,L)$ such that $\mathrm{DF}(\mathcal{X},\mathcal{L}) = 0$. Then $(\mathcal{X},c_1(\mathcal{L}))$ is a relatively K\"ahler cohomological test configuration for $(X,c_1(L))$ such that $\mathrm{DF}(\mathcal{X}, c_1(\mathcal{L})) = 0$. By Theorem \ref{Theorem analytic product conditions} it then follows that its associated geodesic ray is of the form $\exp(tJV).\varphi_0$ for some real holomorphic Hamiltonian vector field $V \in \mathfrak{h}$. By \cite[Lemma 3.2 and Proposition 3.3]{BDL} it then follows that $(\mathcal{X},\mathcal{L})$ is a product configuration. Conversely, by Corollary \ref{Cor product geodesic} product configurations for $(X,\alpha)$ always have vanishing Donaldson-Futaki invariant (since it is K-semistable). 

The second part $(2)$ follows directly from Theorem \ref{Theorem equivalent characterizations discrete}. Indeed, since there are no real holomorphic Hamiltonian vector fields on $X$ it follows that if $\mathrm{DF}(\mathcal{X},\mathcal{A}) = 0$ for a relatively K\"ahler test configuration for $(X,\alpha)$, then the associated geodesic ray is the constant ray. Since the geodesic ray associated to the trivial test coniguration $(X \times \mathbb{P}^1, p_1^*\omega)$ (endowed with the trivial $\mathbb{C}^*$-action) it follows from the key injectivity lemma \ref{Lemma unique model} that $(\mathcal{X},\mathcal{A})$ must be the trivial test configuration. This concludes the proof.   
\end{proof}

\begin{rem} Suppose that $(\mathcal{X},\mathcal{A})$ is a relatively K\"ahler test configuration for $(X,\alpha)$ such that the associated geodesic ray is of the form $\varphi_t = \exp(tJV).\varphi_0$ for some $V \in \mathfrak{h}$ that generates a $\mathbb{C}^*$-action on $X$. 
Then $(\varphi_t)$ is compatible with the product test configuration $(\mathcal{X}_V, \mathcal{A}_V)$
for $(X,\alpha)$ induced by $V$ in the sense of Example \ref{Example product}.
Again, by the injectivity lemma \ref{Lemma unique model} this concludes that $(\mathcal{X},\mathcal{A})$ is $\mathbb{C}^*$-equivariantly isomorphic to the product test configuration $(\mathcal{X}_V,\mathcal{A}_V)$ induced by $V$.    
\end{rem}

\noindent Finally note that the equivalence $(1) \Leftrightarrow (5)$ in Theorem \ref{Theorem analytic product conditions}, combined with \cite[Theorem A]{SD1}, yields the following main result (Theorem \ref{Main geodesic Kps thm}): 

\begin{thm} Suppose that $(X,\omega)$ is a compact K\"ahler manifold. If the K\"ahler class $\alpha := [\omega] \in H^{1,1}(X,\mathbb{R})$ admits a cscK representative, then $(X,\alpha)$ is geodesically K-polystable. 
\end{thm}

\noindent Note that if we specialize to the case of polarized manifolds, then we in particular recover \cite[Theorem 1.6]{BDL}. This is due to  Proposition \ref{Prop comparison polarized} above. It would be interesting to clarify if the above statement is in fact strictly stronger than the usual statement for polarized manifolds involving K-polystability in the sense of Donaldson \cite{Donaldsontoric}. This question relates to the influential examples of \cite{ACGTF}.

\medskip

\subsection{Acknowledgements}
The author would like to thank Julius Ross, Ruadha\'i Dervan, S\'ebastien Boucksom, Vincent Guedj, Jacopo Stoppa, Julien Keller, Eveline Legendre and Paul Gauduchon for helpful discussions related to this work. This work was produced while the author was supported by a PhD grant at Universit\'e Paul Sabatier, Toulouse, joint with \'Ecole Polytechnique, Paris.

\bigskip 

\section{Appendix: Triviality of K\"ahler test configurations - by R. Dervan}

\noindent The goal of this appendix is to clarify some aspects of the triviality of K\"ahler test configurations, using Sj\"ostr\"om Dyrefelt's injectivity result, an idea of Codogni-Stoppa \cite{cs} and some ideas from \cite{Dervanrelative}. In particular this removes the various assumptions on the norms in \cite{Dervanrelative}, thus proving that the existence of an extremal metric implies relative K-stability, and hence also proving that existence of a cscK metric implies equivariant K-polystability.

Fix a K\"ahler manifold $X$ with a K\"ahler class $\alpha$ and let $T \subset \Aut(X,\alpha)$ be a maximal torus, of dimension $d$ say, of automorphisms with Lie algebra $\mft$. We say that a test configuration $(\X,\A,\lambda)$ with $\C^*$-action $\lambda$ is $T$\emph{-invariant} if it admits a $T$-action which commutes with $\pi$ and $\lambda$, extending the action of $T$ on the general fibre. Fix a basis $\beta_1,\hdots \beta_d$ for $T$, and let $\Omega\in\A$ be a relatively K\"ahler metric which is invariant under a maximal compact subgroup of $T$. For each $i$ one can define a family $\varphi^i_t$ of smooth functions by setting $$\beta_i(t)^*\Omega - \Omega = \ddb \varphi^i_t.$$ One can then show that $h_{\beta_i}:=\beta_i(t)_*\dot \varphi(t)$ is a smooth function on $\X$ which is \emph{independent} of $t$ \cite[Definition 2.19]{Dervanrelative}, called (slightly abusively) the \emph{Hamiltonian} of $\beta_i$. By an identical procedure one obtains a Hamiltonian for $\lambda$, denoted $h_{\lambda}$, and one can similarly extend the definition of the Hamiltonian for the flow of arbitrary elements of $\mft$, not just the rational ones (which induce $\C^*$-actions).

\begin{mydef}\cite[Definition 2.19]{Dervanrelative} We define the \emph{inner product} of $\C^*$-actions $\gamma_1,\gamma_2$ on $\X$ (or elements of $\mft$) to be $$\langle \gamma_1,\gamma_2 \rangle = \int_{\X_0} (h_{\gamma_1} - \hat h_{\gamma_1} )(h_{\gamma_2} - \hat h_{\gamma_2}) \Omega_0^n,$$ where $\hat h_{\gamma_1} $ denotes the average value of $h_{\gamma_1}$ on $\X_0$. Similarly we define the \emph{norm} of $(\X,\A,\lambda)$ to be $\|(\X,\A,\lambda)\|_2^2 = \langle \lambda,\lambda\rangle.$\end{mydef}

Recall that to a vector field $\nu \in \mft$ one can associate the Futaki invariant $F(\nu)$. This value is simply the Donaldson-Futaki invariant of the product test configuration when $\nu$ is rational.

\begin{mydef} \label{Appendix def equivariant stability} \cite[Definition 2.23]{Dervanrelative} Choose an orthogonal basis $\beta_i$ for $T$. We define the \emph{relative Donaldson-Futaki invariant} $$\DF_T(\X,\A) = \DF(\X,\A) - \sum_{i=1}^d \frac{\langle \lambda,\beta_i \rangle}{\langle \beta_i,\beta_i \rangle}F(\beta_i).$$ We say that $(X,[\omega])$ is \emph{relatively K-stable} if $\DF_T(\X,\A) \geq 0$ for all $T$\emph{-invariant} test configurations, with equality if and only if $\X_0 \cong X$. We say that $(X,[\omega])$ is \emph{equivariantly K-polystable} if in addition $F(\beta_i)=0$ for all $i$.\end{mydef}

Fixing a $T$-invariant test configuration $(\X,\A,\lambda)$, for any $\C^*$-action $\gamma$ inside $T$ one obtains a test configuration $(\X,\A,\lambda\circ \gamma)$ by composition, and the point of the above definition is that the relative Donaldson-Futaki invariant is invariant under this operation \cite[Section 4]{Dervanrelative}, and hence one can define the Donaldson-Futaki invariant for a composition with any element of $\mft$. Thus one obtains a ``projected test configuration'' $\proj(\X,\A,\lambda) = (\X,\A,\gamma)$ which satisfies the property that $\langle \gamma, \beta_i\rangle = 0$ for all $\beta_i$ (here, strictly speaking, $\gamma$ is the flow of a possibly irrational vector field). It is also clear that one can define the norm for a composition with any element of $\mft$, since the Hamiltonian for such an object makes sense. Let us call such an object an ``irrational test configuration''. The following is a K\"ahler analogue of \cite[Lemma 41]{cs}.

\begin{lem}\label{lemma} We have $\|\proj(\X,\A,\lambda)\|_2 = 0$ if and only if $\X_0 \cong X$. \end{lem}

\begin{proof} If $\proj(\X,\A,\gamma)$ were a genuine test configuration (as is the case when the data is projective), we can follow the argument Codogni-Stoppa use in the projective case. Then  the equality $\|\proj(\X,\A,\lambda)\|_2 = 0$ would imply that $J^{\NA}(\proj(\X,\A,\lambda)) = 0$ by  \cite[Theorem 1.5]{Dervanrelative} ($J^{\NA}$ is called the ``minimum norm'' in \cite{Dervanrelative}). But this would imply $\proj(\X,\A,\lambda)$ were the trivial test configuration by Theorem 1.9, and hence $(\X,\A,\lambda)$ must be a product test configuration, i.e. $\X_0 \cong X$.  

In general we argue as follows. Fix an irrational test configuration $(\X,\A,\gamma)$ and pick some $\Omega \in \A$ a relatively K\"ahler metric which is invariant under a maximal compact subgroup of $T$. One associates a family $\varphi_{\gamma,s}$ of K\"ahler potentials to $\Omega$ in such a way that, when $\gamma$ is rational, $\lim_{s\to \infty} \frac{J(\varphi_{\gamma,s})}{s} = J^{\NA}(\X,\A,\gamma)$. The construction of $\varphi_{\gamma,s}$ implies that $\lim_{s\to \infty} \frac{J(\varphi_{\gamma,s})}{s}$ is a continuous function of $\gamma$, and converges also for arbitrary elements of $\mft$, as follows from the Hamiltonian point of view of \cite{Dervanrelative}. We still denote by $J^{\NA}(\X,\A,\gamma)$ this extension. 


Since the $L^1$-norm is also continuous on $\mft$, and Lipschitz equivalent to $J^{\NA}(\X,\A,\gamma)$ by \cite[Theorem 1.5]{Dervanrelative}, it follows that $J^{\NA}(\proj(\X,\A,\lambda))=0$ and so $J(\varphi_{\gamma,s})$ is bounded independently of $s$. By Lemma 4.8 this in turn implies $J(\psi_{\gamma,s})$ is bounded independently of $s$ where $\psi$ denotes the corresponding ray induced by the geodesic (which is induced from a $L^{\infty}$-relatively K\"ahler metric $\eta$). It follows that the $L^1$-norm of the geodesic is zero, and hence $\psi_s$ is constant.  Therefore $\eta$ extends to a metric on the trivial test configuration by the results of Section 4, so by Theorem 1.9 and the argument of the paragraph above, $\X_0 \cong X$.  \end{proof}

Our main interest in the above is the following corollary of \cite[Theorem 1.2]{Dervanrelative}.

\begin{cor} \label{Appendix main result} If $(X,\alpha)$ admits an extremal metric, then it is relatively K-stable. In particular if $(X,\alpha)$ is cscK, then it is equivariantly K-polystable.  \end{cor}

\begin{proof} \cite[Theorem 1.2]{Dervanrelative} states that in this case $\DF_T(\X,\A,\lambda) \geq 0$, with equality if and only if $\|\proj(\X,\A,\lambda)\|_2 = 0$. But by the lemma above this happens if and only if $\X_0\cong X$, proving the first claim. The second follows as the Futaki invariant vanishes on a cscK K\"ahler manifold.\end{proof}

\begin{rem}

When $X$ is projective, irrational test configurations are closely related to $\R$-degenerations, as used in \cite{ChenSunWang,optimal}, which are simply induced from embedding $X$ into projective space and choosing a real one-parameter subgroup of the automorphism group of that projective space, possibly generated by an irrational vector field. One of the main points of \cite{optimal} is that the majority of the theory works equally well for this generalisation of a (projective) test configuration, and in particular one can define their norm in a natural way. The key point of the proof above is that the norm of an irrational test configuration still has good properties.

\end{rem}

The same proof furnishes the following.

\begin{thm} \label{Appendix geod implies equiv}
Geodesic K-polystability implies equivariant K-polystability.
\end{thm}

\begin{proof}
Suppose $(\X,\A,\lambda)$ is a $T$-invariant test configuration with $\DF(\X,\A,\lambda)=0$. Twisting the action by an element of $\mft$ produces a new irrational test configuration  $(\X,\A,\gamma)$ with constant geodesic. It is then clear that the norm of $(\X,\A,\gamma)$ is zero as the geodesic is constant,  and we conclude by noting that the norm being zero implies $(\X,\A,\gamma)$ is trivial just as in the proof above.\end{proof}

\noindent From Theorem 1.1, this gives another proof that cscK K\"ahler manifolds are equivariantly K-polystable. 

\bigskip

\bibliography{ksemistability} 
\bibliographystyle{amsplain}

\end{document}